%% file: main.tex
\documentclass{article}

\newif\ifarxiv
\arxivtrue

\usepackage{graphicx}
\usepackage{amsmath,amssymb,amsthm,bm}
\usepackage[a4paper]{geometry}
\usepackage{algorithm,algorithmic}
\usepackage[colorlinks=false,pdfborder={0 0 0}]{hyperref}
\usepackage{authblk}
\usepackage{color}
\usepackage{tikz}

\graphicspath{{fig/}}
\DeclareGraphicsExtensions{.mps,.pdf,.eps,.png}
\usetikzlibrary{positioning,calc}

\newenvironment{keywords}{\medskip\textbf{Keywords:}}{}
\newenvironment{AMS}{\medskip\textbf{AMS subject classifications (2020).}}{}

\newtheorem{assumption}{Assumption}
\newtheorem{theorem}{Theorem}

\newtheorem{remark}{Remark}
\newtheorem{lemma}{Lemma}

\newtheorem{property}{Property}
\newtheorem{corollary}{Corollary}
\theoremstyle{plain}

\input macros.tex

\title{Asynchronous Jacobi and randomized Gauss--Seidel methods in shared and distributed memory: A unified convergence-rate analysis}
\author[1]{Erin Carson}
\author[1]{Yuxin Ma}
\affil[1]{Department of Numerical Mathematics, Faculty of Mathematics and Physics, Charles University, Sokolovsk\'{a} 49/83, 186 75 Praha 8, Czechia}

\begin{document}
\maketitle

\input abstract.tex

\input introduction.tex
\input Jacobi_and_GS.tex
\input algorithm.tex

\input convergence.tex

\input discussion.tex
\input conclusion.tex
\input acknowledgments.tex

\bibliographystyle{abbrvurl}
\bibliography{mybib}


\end{document}

%% file: macros.tex
\input macros_latex.tex

\newcommand*{\Ej}[1]{{\mathbb{E}_{j}}}
\newcommand*{\Ejone}[1]{\mathbb{E}_{j+1}}

\newcommand*{\bigO}{O}

\newcommand*{\setphij}[1]{\Phi(#1, d^{(#1)})}
\newcommand*{\xread}[1]{x^{(\setphij{#1})}}

%% file: macros_latex.tex
\DeclareMathOperator{\diag}{diag}

\newcommand{\fro}{\mathsf F}

\newcommand*{\trans}{^{\top}}

\newcommand*{\Etemp}[1]{\widetilde{\mathbb{E}}_{#1}}
\newcommand*{\expectabbr}[1]{\mathbb{E}_{#1}}

\newcommand*{\expect}[1]{\mathbb{E}\bigl[#1\bigr]}

\newcommand*{\abs}[1]{\bigl\lvert#1\bigr\rvert}
\newcommand*{\norm}[1]{\bigl\Vert#1\bigr\rVert}
\newcommand*{\normA}[1]{\Vert#1\rVert_A}
\newcommand*{\normfro}[1]{\bigl\Vert#1\bigr\rVert_{\fro}}

\def\adots{\mathinner{\mkern2mu\raise1pt\hbox{.}\mkern2mu
    \raise4pt\hbox{.}\mkern2mu\raise7pt\hbox{.}\mkern1mu}}

%% file: abstract.tex
\begin{abstract}
Asynchronous iterative methods are attractive for large-scale parallel computing because they reduce synchronization and communication overhead.
Existing convergence-rate analyses, however, have primarily focused on shared memory implementations, whereas distributed memory systems introduce more general and potentially inconsistent communication delays.
In this work, we revisit asynchronous Jacobi and randomized Gauss--Seidel (RGS) methods for symmetric positive definite linear systems from a unified perspective.
We first introduce a general asynchronous model that encompasses both shared memory and distributed memory settings and expresses the two methods through a common coordinate-update framework.
We then establish linear convergence in expectation under an explicit stability condition.
The resulting convergence bound depends algebraically on the delay through the quantity \(\sqrt{\rho\tau}+\rho\tau\), where \(\tau\) is the maximum communication delay and \(\rho\) reflects the communication pattern of the underlying parallel implementation.
In particular, under an appropriate scaling regime with \(\rho\tau=O(1)\), the guaranteed per-iteration convergence rate has the same asymptotic order as that of synchronous RGS.
These results provide a unified framework for quantifying the effect of asynchronicity on asynchronous Jacobi/RGS methods in both shared and distributed memory environments.


\end{abstract}

%% file: introduction.tex
\section{Introduction}
\label{sec:introduction}
We consider the problem of solving large scale linear systems of the form \(Ax = b\), which arise in a wide range of scientific and engineering applications, including numerical solutions of partial differential equations, optimization, and data analysis.
A classical approach to solving such systems is based on stationary iterative methods, among which the Jacobi and Gauss--Seidel (GS) algorithms are two of the most fundamental and widely used schemes.
These methods are particularly attractive due to their simplicity, low per-iteration computational cost, and suitability for large sparse problems, making them essential building blocks in modern scientific computing.

However, in parallel systems, the performance and scalability of these iterative methods are often limited by the high cost of communication.
In particular, traditional implementations rely on synchronous execution, where global synchronization points are required to ensure that all processing units have completed their updates before proceeding to the next iteration.
Such synchronization points can lead to significant idle time and communication overhead, especially in heterogeneous or high-latency systems.
To address this issue, asynchronous iterative methods have been proposed, in which processors perform updates using the most recently available information without waiting for global synchronization.

In recent decades, asynchronous iterative methods have attracted attention due to their ability to exploit parallelism without the need for strict synchronization.
Early work in this area can be traced back to chaotic relaxation~\cite{CM1969}, where updates are performed using possibly outdated information.
Building on this idea, a variety of asynchronous algorithms have been developed for solving linear systems, including asynchronous Jacobi~\cite{BBDH2014,TNCA2022,WC2018,WC2019}, asynchronous Gauss--Seidel~\cite{ADG2015}, asynchronous Richardson~\cite{CFS2021}, and asynchronous Schwarz~\cite{FSS1997,MSV2017,NCA2021,YCBD2019} methods. 
These approaches are particularly well suited for modern large scale computing environments, such as shared memory and distributed memory systems, where communication delays and lack of synchronization are inherent.

Despite the broad development of asynchronous methods, most existing theoretical studies have primarily focused on establishing convergence guarantees, i.e., whether an asynchronous algorithm converges under certain assumptions~\cite{CFS2021,FSS1997,MSV2017,WC2018,WC2019}.
Significantly fewer works provide explicit descriptions of convergence rates.
In particular, \cite{ADG2015} establishes convergence rates for the asynchronous randomized Gauss--Seidel (RGS) method in shared memory systems, based on the convergence analysis of RGS from~\cite{LL2010}.
However, distributed memory implementations involve different communication patterns, where locally available components can be current while remotely communicated components may be delayed.
Existing shared memory models and their convergence analyses do not directly capture this setting.
This motivates the development of a convergence-rate theory that can account for both shared memory and distributed memory implementations within a unified framework.

In this paper, we address this gap by studying asynchronous Jacobi and RGS methods in shared memory and distributed memory environments within a unified framework.
We first introduce a unified asynchronous model that captures the different communication patterns arising in these two parallel architectures.
We then derive a convergence-rate theory for this unified model that establishes linear convergence in expectation for symmetric positive definite linear systems. 
The analysis quantifies how the convergence rate is affected jointly by the communication delay and the underlying communication pattern, and shows that the resulting delay dependence is algebraic rather than exponential.
Finally, our results identify a regime in which asynchronicity does not change the asymptotic order of the convergence rate relative to synchronous RGS.

The remainder of this paper is organized as follows.
In Section~\ref{sec:preliminary}, we review the classical synchronous Jacobi and Gauss--Seidel methods.
Section~\ref{sec:algorithm} introduces asynchronous Jacobi algorithms in both shared memory and distributed memory settings, and discusses their equivalence to asynchronous RGS.
Furthermore, we propose a unified modeling framework.
Under this framework, Section~\ref{sec:conv} presents our convergence rate analysis for asynchronous algorithms.
Finally, in Section~\ref{sec:discussion}, we provide numerical experiments that simulate asynchronous algorithms to validate and illustrate the theoretical results.

Before proceeding, we summarize the notation and some basic definitions used throughout the paper.
We adopt MATLAB-style notation for submatrices.
For example, \(x_k\) denotes the \(k\)-th component of the vector \(x\), and \(A_{\Omega, :}\) represents the collection of all rows of the matrix \(A\) whose indices are contained in the set \(\Omega\).
In addition, we write \(I_{i, :}\) simply as \(I_i\) to indicate the \(i\)-th column of the \(n\)-by-\(n\) identity matrix \(I\).
We use \(\norm{\cdot}\) for the \(2\)-norm, \(\normfro{\cdot}\) for the Frobenius norm, and \(\normA{\cdot}\) for the \(A\)-norm associated with a symmetric positive definite matrix \(A\).
The condition number \(\kappa(A)\) is defined as \(\sigma_{\max}(A)/\sigma_{\min}(A)\), where \(\sigma_{\max}(A)\) and \(\sigma_{\min}(A)\) denote the largest and smallest singular values of \(A\), respectively.
We use \(\lambda_{\max}\) and \(\lambda_{\min}\) to denote the largest and smallest eigenvalues, respectively.

%% file: Jacobi_and_GS.tex
\section{Jacobi and Gauss--Seidel methods}
\label{sec:preliminary}
In this section, we briefly review the Jacobi and GS methods.
We split \(A\) as \(A = \Lambda - L - U\), where \(\Lambda := \diag(A)\) contains the diagonal elements of \(A\), \(L\) denotes the strictly lower triangular part of \(A\), and \(U\) denotes the strictly upper triangular part.
To align with the notation used in the following sections, we assume that only one entry of \(x\) is updated in each iteration.
The iteration formulas for the Jacobi and GS methods (with successive over-relaxation) are
\begin{equation}
\begin{split}
    \text{Jacobi method:}\quad& x^{\bigl((i+1)n\bigr)} = \Lambda^{-1}(L+U)x^{(in)} + \Lambda^{-1} b, \\ 
    \text{GS method:}\quad& x^{\bigl((i+1)n\bigr)} = (\Lambda - \beta L)^{-1}\bigl((1-\beta)\Lambda + \beta U\bigr)x^{(in)} + \beta(\Lambda - \beta L)^{-1} b,
\end{split}
\end{equation}
where \(\beta\) is the relaxation factor.

For each iteration, the update can be expressed as
\begin{equation} \label{eq:component-form}
\begin{split}
    x^{(j+1)} &= x^{(j)} + \frac{\beta\, d^{(j)}\bigl(d^{(j)}\bigr)\trans (b - A\xread{j})}{\bigl(d^{(j)}\bigr)\trans Ad^{(j)}} \\
    &= x^{(j)} + \frac{\beta d^{(j)}\bigl(d^{(j)}\bigr)\trans A(x^\star - \xread{j})}{\bigl(d^{(j)}\bigr)\trans \Lambda d^{(j)}},
\end{split}
\end{equation}
where \(Ax^\star=b\), \(d^{(j)}\in \{I_1, \dotsc, I_n\}\) is the chosen update direction in the \((j+1)\)-st iteration, and \(\xread{j}\) denotes the data currently available for updating \(x^{(j+1)}\).
The parameters \(\beta\), \(d^{(j)}\), and \(\xread{j}\) are specified by
\begin{equation} \label{eq:rev:Jacobi-GS}
\begin{split}
    \text{Jacobi method:}\quad& \beta = 1,\quad d^{(j)} = I_{\mathrm{mod}\,(j+1, n)}, \quad
    \xread{j} = x^{(j-\mathrm{mod}\,(j+1, n)+1)},\\ 
    \text{GS method:}\quad& d^{(j)} = I_{\mathrm{mod}\,(j+1, n)}, \quad \xread{j} = x^{(j)},
\end{split}
\end{equation}
respectively, where \(\mathrm{mod}(\cdot, n)\) denotes the congruence modulo a given integer \(n\).

In~\eqref{eq:rev:Jacobi-GS}, for the GS method, each component of \(x\) is updated in a fixed and deterministic order from \(1\) to \(n\), always using the latest available values.
In contrast, the RGS method selects an update direction \(d^{(j)}\in\{I_1, I_2, \dotsc, I_n\}\) at random with uniform probability; that is, it uses~\eqref{eq:component-form} with
\begin{equation}
    \text{RGS method:}\quad d^{(j)}\sim \mathcal{U}(I_1, \dotsc, I_n),\quad \xread{j} = x^{(j)}.
\end{equation}

%% file: algorithm.tex
\section{Algorithms and a unified model for asynchronous methods}
\label{sec:algorithm}
In this section, we describe the asynchronous Jacobi method on different parallel architectures, and then propose a unified model for asynchronous Jacobi that applies across these diverse architectures.

\subsection{Algorithms on shared memory and distributed memory systems}
Algorithms~\ref{algo:async-Jacobi-shared} and~\ref{algo:async-Jacobi} present the asynchronous Jacobi method for different parallel architectures, specifically for shared memory and distributed memory systems, respectively.
For the sake of analysis, we impose an order \(x^{(0)}, x^{(1)}, \dotsc\) on the values of \(x\) during the computation.  
Here, \(x^{(j)}\) denotes the value of \(x\) after \(j\) updates, with each update modifying exactly one component of \(x\), as shown in~\eqref{eq:component-form}; that is, it corresponds to one loop of Algorithm~\ref{algo:async-Jacobi-shared} or~\ref{algo:async-Jacobi} executed by one processor or node.

\begin{remark}[Equivalence between asynchronous Jacobi and asynchronous RGS methods] \label{remark:AsyJ=AsyRGS}
    In~\cite[Algorithm~1]{ADG2015}, the authors present an algorithm for asynchronous RGS on shared memory systems.
    Comparing this algorithm with the asynchronous Jacobi method given in Algorithm~\ref{algo:async-Jacobi-shared}, we see that the only distinction lies in the selection of \(k\).
    In \cite[Algorithm~1]{ADG2015}, each processing unit independently chooses a random \(k\), uniformly distributed over \(\{1, 2, \dotsc, n\}\), at the beginning of each iteration.
    By contrast, Algorithm~\ref{algo:async-Jacobi-shared} selects \(k\) according to a prescribed deterministic order.
    However, in a parallel environment, even if each processing unit iteratively updates the components of \(x\) in a specific order, the overall order in which the components of \(x\) are updated across all processing units is random.
    In this sense, the algorithm for asynchronous Jacobi is the same as the algorithm for asynchronous RGS.
\end{remark}

\begin{algorithm}[!tb]
\begin{algorithmic}[1]
    \caption{Asynchronous Jacobi method on shared memory systems for each processor}
    \label{algo:async-Jacobi-shared}
    \REQUIRE
     A matrix \(A \in \mathbb R^{n\times n}\), the right-hand side \(b\in \mathbb R^{n}\), an initial guess \(x\), and \(\beta\in (0, 2)\).
    \ENSURE
    Computed vector \(x\) approximating \(Ax=b\).
    
    \WHILE{Not converge}
        \FOR{\(k = 1\), \(2\), \(\dotsc\), \(n\)}
            \STATE Update \(\gamma\gets b_k - A_{k, :}x\) by recent \(x_j\)'s read from shared memory.
            \STATE \(x_k \gets x_k + \beta\gamma/A_{k,k}\).
        \ENDFOR
    \ENDWHILE
\end{algorithmic}
\end{algorithm}

\begin{algorithm}[!tb]
\begin{algorithmic}[1]
    \caption{Asynchronous Jacobi method on distributed memory systems for each node}
    \label{algo:async-Jacobi}
    \REQUIRE
     A matrix \(A \in \mathbb R^{n\times n}\), the right-hand side \(b\in \mathbb R^{n}\), an initial guess \(x\), and \(\beta\in (0, 2)\).
     The \(i\)-th node owns \(x_{\Omega_i}\), \(b_{\Omega_i}\), and \(A_{\Omega_i, :}\), where there are \(p\) nodes and \(\{\Omega_i\}_{i=1}^p\) is a partition of the set \(\{1, 2, \dotsc, n\}\), i.e., \(\cup_{i = 1}^p \Omega_i = \{1, 2, \dotsc, n\}\).
    \ENSURE
    Computed vector \(x\) approximating \(Ax=b\).
    
    \WHILE{Not converge}
        \FOR{\(k = 1\), \(2\), \(\dotsc\), \(n\)}
        \IF{\(k\in \Omega_i\)}
            \STATE Update \(\gamma\gets b_k - A_{k, :}x\) by recent \(x_j\)'s received from all nodes.
            \STATE \(x_k \gets x_k + \beta\gamma/A_{k,k}\).
            \STATE Send \(x_k\) to every node \(j\) for which \(A_{\Omega_j,k}\neq 0\).
        \ENDIF
        \ENDFOR
    \ENDWHILE
\end{algorithmic}
\end{algorithm}

\subsection{A unified model for shared memory and distributed memory systems}
We consider two parallel computing environments with \(p\) processing units: a shared memory system with \(p\) processors and a distributed memory system with \(p\) nodes.
In the shared memory setting, all processors have access to a common address space, in which the entire matrix \(A\) and the vectors \(b\), \(x\) are stored.
Each processor can read from and write to this shared memory.
In the distributed memory setting, the data is partitioned across the \(p\) nodes.
Specifically, each node owns a subset of the rows of the matrix \(A\), along with the corresponding part of \(b\) and \(x\).

Using~\eqref{eq:component-form}, the \((j+1)\)-st iteration of the asynchronous Jacobi method can be expressed as
\begin{equation} \label{eq:iter-AsyJ}
\begin{split}
    x^{(j+1)} &= x^{(j)} + \frac{\beta d^{(j)}\gamma^{(j)}}{\bigl(d^{(j)}\bigr)\trans \Lambda d^{(j)}}
    \quad\text{with}\quad \gamma^{(j)} = \bigl(d^{(j)}\bigr)\trans A \bigl(x^\star - \xread{j}\bigr),
\end{split}
\end{equation}
where \(\Lambda = \diag(A)\) as defined in Section~\ref{sec:preliminary}, \(d^{(j)}\in \{I_1, \dotsc, I_n\}\) represents the update direction, and \(\xread{j}\) denotes the data currently available for updating \(\gamma^{(j)}\) in the \((j+1)\)-st iteration of Algorithms~\ref{algo:async-Jacobi-shared} and~\ref{algo:async-Jacobi}.
We then impose assumptions on \(d^{(j)}\), the delay bound, and \(\xread{j}\), respectively.

\begin{assumption} \label{assmp:equal}
    Each element of \(x\) is equally likely to be updated, i.e., the update direction \(d^{(j)}\) satisfying \(d^{(j)} \sim \mathcal{U}(I_{1}, \dotsc, I_{n})\).
\end{assumption}

This assumption imposes a specific random distribution over the update direction \(d^{(j)}\), even though, in practice, we do not explicitly choose a random direction uniformly from \(\{I_1, \dotsc, I_n\}\).
The same assumption is also used in the work on RGS~\cite{LL2010} and asynchronous RGS~\cite{ADG2015}.

\begin{assumption} \label{assmp:delay_bound}
    There exists a constant \(\tau\) such that, for every iteration \(j = 1, 2, \dotsc\), all updates that are at least \(\tau\) iterations old are included in the computation of iteration \(j\).
    Also, we assume that \(\tau\) does not depend on the choices \(d^{(j)}\).
    In particular, for the synchronous algorithm we have \(\tau = 0\).
\end{assumption}

The second assumption requires that the asynchronicity is bounded, which is also commonly adopted in other works analyzing provable convergence rates.

With Assumptions~\ref{assmp:equal} and~\ref{assmp:delay_bound}, the governing iteration of asynchronous Jacobi (or asynchronous RGS as discussed in Remark~\ref{remark:AsyJ=AsyRGS}) executed by all processing units is
\begin{equation} \label{eq:model}
\begin{split}
    d^{(j)} &\sim U(I_{1}, \dotsc, I_{n}), \\
    x^{\bigl(\setphij{j}\bigr)} &= x^{(0)} + \sum_{k\in\setphij{j}} \frac{\beta  d^{(k)}\gamma^{(k)}}{\bigl(d^{(k)}\bigr)\trans \Lambda d^{(k)}}, \\
    x^{(j+1)} &= x^{(j)} + \frac{\beta d^{(j)}\gamma^{(j)}}{\bigl(d^{(j)}\bigr)\trans \Lambda d^{(j)}}
    \quad\text{with}\quad \gamma^{(j)} = \bigl(d^{(j)}\bigr)\trans A \bigl(x^\star - \xread{j}\bigr),
\end{split}
\end{equation}
where the choice of \(\setphij{j}\) is determined by the specific parallel architecture.

Next, the only remaining unspecified quantity in~\eqref{eq:model} is
\(\xread{j}\). We therefore turn to the construction of \(\xread{j}\), or,
equivalently, to the definition of the index set \(\setphij{j}\).
To characterize the information that must be available when \(\xread{j}\) is
formed, we introduce an auxiliary matrix \(D^{(j)}\). The matrix \(D^{(j)}\)
identifies the components of \(x^{(j)}\) whose most recent updates must be
accessible at iteration \(j\). More precisely, \(\setphij{j}\) is required to
contain all iteration indices associated with the updates contributing to
\(\bigl(D^{(j)}\bigr)\trans x^{(j)}\). Let \(\Psi(D^{(j)})\) denote the set of
these iteration indices. We define
\[
    \setphij{j}
    =
    \tilde{\Phi}(j)\cup\Psi(D^{(j)}),
\]
where
\[
    \{0,1,\dotsc,\max\{0,j-\tau\}\}
    \subseteq
    \tilde{\Phi}(j)
    \subseteq
    \{0,1,\dotsc,j-1\},
\]
and \(\tilde{\Phi}(j)\) is independent of \(d^{(j)}\). Consequently,
\[
    \{0,1,\dotsc,\max\{0,j-\tau\}\}\cup\Psi(D^{(j)})
    \subseteq
    \setphij{j}
    \subseteq
    \{0,1,\dotsc,j-1\}.
\]
By construction, \(\xread{j}\) contains the most recent updates of all
components selected by \(D^{(j)}\), and hence
\begin{equation} \label{eq:xphitDj=xjtDj}
    \bigl(D^{(j)}\bigr)\trans \xread{j}
    =
    \bigl(D^{(j)}\bigr)\trans x^{(j)}.
\end{equation}

For the convergence analysis in Section~\ref{sec:conv}, we assume that
\(\setphij{j}\) is independent of the choices of
\(d^{(j-\tau)},\dotsc,d^{(j-1)}\). In practice, this assumption may not hold
exactly, because the availability of recent updates may depend on the selected
update directions through variations in communication efficiency among
different processing units.

The choice of \(D^{(j)}\) depends on the underlying communication architecture.
For a distributed memory system without local communication, a natural choice
is
\[
    D^{(j)}\in\{I_{\Omega_1},\dotsc,I_{\Omega_p}\}.
\]
In particular, \(D^{(j)}\) can be chosen such that
\[
    \bigl(d^{(j)}\bigr)\trans D^{(j)}\neq 0,
\]
because the node responsible for the \((j+1)\)-st iteration also owns the
components represented by
\(\bigl(D^{(j)}\bigr)\trans x^{(j)}\). Therefore, the most recent updates of
these components are always locally accessible when \(\xread{j}\) is formed,
and~\eqref{eq:xphitDj=xjtDj} holds. When local communication is allowed, the
subdomains represented by the admissible choices of \(D^{(j)}\) may overlap.

For shared memory systems, as discussed in
Remark~\ref{remark:AsyJ=AsyRGS}, the asynchronous Jacobi method coincides with
the asynchronous RGS method. In~\cite[Equation~(9)]{ADG2015}, the authors set
\(\setphij{j}=\tilde{\Phi}(j)\) and proposed a model for asynchronous RGS under
the assumption \(\Lambda=I\), so that
\[
    \bigl(d^{(j)}\bigr)\trans\Lambda d^{(j)}=1
\]
for every \(d^{(j)}\). This setting is recovered as a special case of the
unified model~\eqref{eq:model} by taking
\[
    \Psi(D^{(j)})=\emptyset,
    \qquad
    \setphij{j}=\tilde{\Phi}(j).
\]

%% file: convergence.tex
\section{Convergence analysis}
\label{sec:conv}
Existing convergence-rate theory for asynchronous RGS was developed by~\cite{ADG2015} for shared memory systems.
A direct extension of that analysis to the unified model~\eqref{eq:model}, however, encounters two difficulties.
First, the resulting bounds contain factors that grow exponentially with the delay parameter \(\tau\).
Such a dependence is particularly unfavorable for distributed memory systems, where communication delays may be much larger and hence \(\tau\) can be large.
Second, in the unified model the stale vector \(\xread{j}\) may depend on the specific choice of the current update direction \(d^{(j)}\).
Consequently, some expectation arguments used in the shared memory analysis cannot be carried over directly.
These two issues prevent a straightforward extension of the analysis of Avron et al. to the present unified model and motivate the different approach developed below.
The resulting convergence bound therefore has a substantially different form and avoids factors that grow exponentially in \(\tau\).

We next specify the communication schedule used in the probabilistic formulation.
Let
\begin{equation} \label{eq:def-schedule-new}
    \mathcal S:=\{\tilde\Phi(j)\}_{j\geq0}
\end{equation}
denote the complete baseline communication schedule, where every realization of \(\mathcal S\) satisfies Assumption~\ref{assmp:delay_bound}.
We assume that the complete sequence of update directions \(\{d^{(j)}\}_{j\geq0}\) is i.i.d. uniform and independent of \(\mathcal S\).
This assumption will be used explicitly in the proof in Section~\ref{subsec:proof-improved}.

To simplify the statement of the convergence result, we introduce
\begin{equation} \label{eq:def-Ej-Pj}
    \expectabbr{j}:=\expect{\normA{x^{(j)}-x^\star}^2}.
\end{equation}
Let \(\bar A:=\sqrt{\Lambda}^{-1}A\sqrt{\Lambda}^{-1}\), and define \(\bar A_{\rm off}=\bar A\) for the shared memory model and, for the distributed memory model, define \(\bar A_{\rm off}\) as the block off-diagonal part of \(\bar A\), namely
\begin{equation} \label{eq:def-Aoff-new}
    \begin{cases}
        \bigl(I_{\Omega_r}\bigr)\trans \bar A_{\rm off}I_{\Omega_t}
        =\bigl(I_{\Omega_r}\bigr)\trans \bar A I_{\Omega_t}, & r\neq t,\\
        \bigl(I_{\Omega_r}\bigr)\trans \bar A_{\rm off}I_{\Omega_t}=0, & r=t.
    \end{cases}
\end{equation}
Define
\begin{equation} \label{eq:def-rho-omega}
    \rho:=\frac1n\max_{t=1,\dotsc,n}\sum_{r=1}^n\abs{(\bar A_{\rm off})_{r,t}},
    \qquad
    \omega_\tau:=\sqrt{\rho\tau}+\rho\tau.
\end{equation}
The parameter \(\rho\) measures the strength of the coupling associated with the entries of \(\bar A\) that may be affected by stale information.
For the shared memory model, \(\bar A_{\rm off}=\bar A\), and hence \(\rho\) depends only on the coefficient matrix.
For the distributed memory model, \(\bar A_{\rm off}\) contains only the off-diagonal blocks associated with different nodes, and therefore \(\rho\) depends on both the coefficient matrix and the partition of the data among the nodes.
In particular, a smaller \(\rho\) indicates weaker coupling through potentially stale information and, as will be seen in Theorem~\ref{thm:improved}, leads to a weaker dependence of the convergence bound on the delay parameter \(\tau\).

\begin{theorem} \label{thm:improved}
    Let \(A\in\mathbb R^{n\times n}\) be symmetric positive definite and then define \(\bar A=\sqrt{\Lambda}^{-1}A\sqrt{\Lambda}^{-1}\).
    Suppose that the asynchronous Jacobi/RGS iteration satisfies the unified model~\eqref{eq:model} and Assumption~\ref{assmp:delay_bound}.
    Assume that the update directions are i.i.d. uniform and independent of the baseline communication schedule \(\mathcal S\) defined in~\eqref{eq:def-schedule-new}.
    Let \(\tau\geq1\), and define \(\rho\) and \(\omega_\tau\) by~\eqref{eq:def-rho-omega}.
    Assume that
    \begin{equation} \label{eq:beta-condition-new}
        2-\beta-2\beta\omega_\tau>0.
    \end{equation}
    Then, for every integer \(l_0\geq1\) and \(t\geq0\),
    \begin{equation} \label{eq:thm-improved:convrate-new}
        \expectabbr{t(\tau+l_0)}
        \leq \left(1-\frac{a_\tau c_{l_0,\tau}}
        {1+\beta\bigl(2-\beta-\beta\omega_\tau+\beta\tau\bigr)c_{l_0,\tau}
        +2\beta\sqrt{\tau c_{l_0,\tau}}} \right)^t
        \expectabbr{0},
    \end{equation}
    where
    \begin{equation} \label{eq:def-a-new}
        a_\tau:=\beta\bigl(2-\beta-2\beta\omega_\tau\bigr)>0,
    \end{equation}
    and, with \(\mu:=\lambda_{\min}(\bar A)/n\),
    \begin{equation} \label{eq:def-cl0tau-new}
        c_{l_0,\tau}
        := \frac{\mu l_0}
        {\left[
        1+\beta\sqrt{\mu l_0(\tau+l_0-1)}+\beta\omega_\tau
        \right]^2}.
    \end{equation}
\end{theorem}


In the following, we first provide a discussion in Section~\ref{subsec:discussion-improved}, followed by a proof of this theorem in Section~\ref{subsec:proof-improved}.

\subsection{Discussion of Theorem~\ref{thm:improved}}
\label{subsec:discussion-improved}
Theorem~\ref{thm:improved} establishes a linear convergence rate over blocks of \(\tau+l_0\) iterations.
The effect of asynchronicity is characterized by
\(\omega_\tau=\sqrt{\rho\tau}+\rho\tau\), rather than by factors that grow exponentially with \(\tau\).
In particular, if \(\rho\tau=\bigO(1)\), then \(\omega_\tau=\bigO(1)\).
Therefore, the effect of stale information remains bounded as the delay increases, provided that the coupling parameter \(\rho\) decreases proportionally to \(1/\tau\).
In particular, the choice \(l_0=\tau\) leads to a simple explicit contraction over blocks of \(2\tau\) iterations, as summarized in the following corollary.

\begin{corollary} \label{cor:fixedl0-new}
    Under the assumptions of Theorem~\ref{thm:improved}, suppose in addition that
    \begin{equation} \label{cor:eq:assump1}
        \beta^2\mu\tau^2\leq\frac12
    \end{equation}
    and
    \begin{equation} \label{cor:eq:assump2}
        \beta\bigl(2-\beta-\beta\omega_\tau+\beta\tau\bigr)
        \mu\tau\leq1.
    \end{equation}
    Then, by choosing \(l_0=\tau\),
    \begin{equation} \label{eq:cor-improved:convrate-new}
        \expectabbr{2t\tau}
        \leq \left(1-\frac{a_\tau\mu\tau}{36}\right)^t \expectabbr{0}
        \quad\text{with}\quad
        a_\tau=\beta\bigl(2-\beta-2\beta\omega_\tau\bigr).
    \end{equation}
\end{corollary}

\begin{proof}
    Since \(l_0=\tau\),
    \[
        c_{\tau,\tau}
        = \frac{\mu\tau}
        {\left(1+\beta\sqrt{\mu\tau(2\tau-1)}
        +\beta\omega_\tau \right)^2}.
    \]
    The assumption~\eqref{eq:beta-condition-new} implies \(\beta\omega_\tau<1\), while the assumption~\eqref{cor:eq:assump1} gives
    \begin{equation} \label{eq:cor-proof:beta-sqrt-mubeta}
        \beta\sqrt{\mu\tau(2\tau-1)}
        \leq \sqrt{2\beta^2\mu\tau^2}
        \leq 1.
    \end{equation}
    Hence
    \begin{equation} \label{eq:cor-proof:bound-c}
        c_{\tau,\tau}\geq\frac{\mu\tau}{9}.
    \end{equation}
    Moreover, since \(c_{\tau,\tau}\leq\mu\tau\), \eqref{cor:eq:assump2} implies
    \[
        \beta\bigl(2-\beta-\beta\omega_\tau+\beta\tau\bigr)
        c_{\tau,\tau}\leq1
    \]
    and~\eqref{eq:cor-proof:beta-sqrt-mubeta} implies
    \[
        2\beta\sqrt{\tau c_{\tau,\tau}}
        \leq \sqrt{2}\cdot\sqrt{2\beta^2\mu\tau^2}
        \leq\sqrt2<2.
    \]
    Thus, the denominator in~\eqref{eq:thm-improved:convrate-new} is smaller than \(4\).
    Together with~\eqref{eq:cor-proof:bound-c}, the result follows directly from Theorem~\ref{thm:improved}.
\end{proof}

\begin{remark}[Asymptotic interpretation of the convergence rate]
\label{rem:asymptotic-new}
    Suppose that \(\rho\tau=\bigO(1)\) and that the assumptions of Corollary~\ref{cor:fixedl0-new} hold.
    Since \(\omega_\tau=\sqrt{\rho\tau}+\rho\tau=\bigO(1)\), a natural choice of the relaxation parameter can be obtained by maximizing the coefficient
    \[
        a_\tau(\beta)
        = \beta\bigl(2-\beta-2\beta\omega_\tau\bigr).
    \]
    This gives
    \begin{equation} \label{eq:beta-star-new}
        \beta^\star
        = \frac{1}{1+2\omega_\tau},
        \qquad a_\tau(\beta^\star)
        = \frac{1}{1+2\omega_\tau}.
    \end{equation}
    In particular, \(\beta^\star=\Theta(1)\) and \(a_\tau(\beta^\star)=\Theta(1)\) when \(\rho\tau=\bigO(1)\).

    With \(l_0=\tau\), Corollary~\ref{cor:fixedl0-new} gives a contraction factor over \(2\tau\) iterations of the form 
    \[
        1-\Theta\bigl(a_\tau\mu\tau\bigr).
    \]
    Since the block length is \(2\tau\), the corresponding effective per-iteration contraction factor is
    \[
        \left(1-\Theta\bigl(a_\tau\mu\tau\bigr)\right)^{1/(2\tau)}
        = e^{-\Theta(a_\tau\mu\tau)}
        = 1-\Theta\bigl(a_\tau\mu\bigr).
    \]
    For the choice \(\beta=\beta^\star\), this becomes
    \[
        1-\Theta\left(\frac{\mu}{1+2\omega_\tau}\right)
        = 1-\Theta\left(
        \frac{\lambda_{\min}(\bar A)}
        {n\bigl(1+2\sqrt{\rho\tau}+2\rho\tau\bigr)}
        \right).
    \]
    Hence, when \(\rho\tau=\bigO(1)\), the guaranteed asynchronous rate has the same asymptotic order in \(n\) as the synchronous RGS rate.
\end{remark}

\subsection{Proof of Theorem~\ref{thm:improved}}
\label{subsec:proof-improved}
We now prove Theorem~\ref{thm:improved}.
Throughout this subsection, we fix an arbitrary realization of the communication schedule \(\mathcal S\) that satisfies the bounded-delay assumption.
Conditional on this fixed schedule, the baseline read sets \(\tilde\Phi(j)\), the missing-update sets
\begin{equation} \label{eq:def-Phi-minus-new}
    \tilde{\Phi}^{-}(j)
    :=\{k<j:k\notin\tilde\Phi(j)\}
    =\{0,1,\dotsc,j-1\}\setminus\tilde\Phi(j),
\end{equation}
and the future missing counts \(m^{(j)}(k)\) introduced below are deterministic.
The only randomness in the following analysis therefore comes from the random update directions \(d^{(0)},d^{(1)},\dotsc\), which remain independent and uniformly distributed after conditioning on \(\mathcal S\).
By Assumption~\ref{assmp:delay_bound},
\[
    \tilde{\Phi}^{-}(j)
    \subseteq\{\max\{0,j-\tau+1\},\dotsc,j-1\},
    \qquad
    \abs{\tilde{\Phi}^{-}(j)}\leq\tau.
\]
For simplicity, we suppress the conditioning on \(\mathcal S\) from the notation.
Since all estimates below are uniform over admissible realizations of \(\mathcal S\), the unconditional bounds follow by subsequently averaging over the communication schedule.
Notice that the iteration of Algorithm~\ref{algo:async-Jacobi}, i.e., \eqref{eq:model}, is mathematically equivalent to
\begin{equation}
\begin{split}
    \bar{\gamma}^{(j)} &= \bigl(d^{(j)}\bigr)\trans \bar{A}\bigl(\bar{x}^\star - \bar{x}^{(\setphij{j})}\bigr) = \bigl(d^{(j)}\bigr)\trans\sqrt{\Lambda}d^{(j)}\gamma^{(j)}, \\
    \bar{x}^{(j+1)} &= \bar{x}^{(j)} + \beta d^{(j)}\bar{\gamma}^{(j)} = \sqrt{\Lambda}x^{(j+1)}
\end{split}
\end{equation}
with \(\bar{A} = \sqrt{\Lambda}^{-1}A\sqrt{\Lambda}^{-1}\) and \(\bar{A}\bar{x}^\star = \sqrt{\Lambda}^{-1}b\), which solves \(\bar{A}\bar{x}=\sqrt{\Lambda}^{-1}b\).
Together with the definition of \(\bar{A}\), we have
\begin{equation}
    \normA{x^{(j)} - x^\star}^2
    = \norm{\bar{x}^{(j)} - \bar{x}^\star}_{\bar{A}}^2
    \quad\text{and}\quad
    \normA{\xread{j} - x^\star}^2
    = \norm{\bar{x}^{(\setphij{j})} - \bar{x}^\star}_{\bar{A}}^2.
\end{equation}
For \(\bar{A} = \sqrt{\Lambda}^{-1} A \sqrt{\Lambda}^{-1}\), where \(A\) is symmetric positive definite, it follows directly that \(\bar{A}\) is also symmetric positive definite and has unit entries on its diagonal.
Consequently, every second-order leading principal minor of \(\bar{A}\) is strictly positive, which yields \(\bar{A}_{i,j} < 1\) with \(i\neq j\) and therefore \(\norm{\bar{A}} \le \normfro{\bar{A}} < n\).
Thus, \textit{without loss of generality, we assume that \(A_{k,k} = 1\) for all \(1\leq k\leq n\) in the following proof.}
We likewise write \(A_{\rm off}\) for \(\bar A_{\rm off}\), set \(A_d:=A-A_{\rm off}\), and keep the same symbols \(\rho\) and \(\mu=\lambda_{\min}(A)/n\).
For the distributed memory model, \(A_d\) is the block diagonal part of \(A\); for the shared memory model, \(A_d=0\).

Before giving the details, we briefly summarize the proof strategy as follows:
\begin{enumerate}
\item Step~1 derives an estimate for the error caused by stale information over a group of consecutive iterations.
\item Step~2 uses this estimate to construct a modified error quantity that decreases at every iteration despite the presence of delayed updates.
\item Step~3 shows that a sufficient amount of progress must be made over a block of \(\tau+l_0\) iterations.
\item Step~4 combines these results to obtain a contraction over each block.
\item Finally, Step~5 optimizes the resulting bound and proves Theorem~\ref{thm:improved}.
\end{enumerate}

\paragraph{Step 1: Controlling the effect of stale information.}
The main difficulty caused by asynchronicity is that an update may be computed using values from different past iterations.
We first express the discrepancy caused by the missing updates and then estimate its contribution to the residual.
Let
\begin{equation} \label{eq:def-e-z-g}
    e^{(j)}:=x^\star-x^{(j)}.
\end{equation}
Since
\[
    x^{(j)}=x^{(0)}+\beta\sum_{k=0}^{j-1}d^{(k)}\gamma^{(k)},
\]
the updates contained in the current iterate but missing from the baseline read are exactly those indexed by \(\tilde{\Phi}^{-}(j)\).
We therefore define
\begin{equation} \label{eq:z-missing-sum}
    z^{(j)}:=\beta\sum_{k\in\tilde{\Phi}^{-}(j)}d^{(k)}\gamma^{(k)}.
\end{equation}
Thus, the value used for the potentially stale part of the residual is the current iterate minus \(z^{(j)}\).
For the distributed memory model, the entries corresponding to \(A_d\) are current, whereas those corresponding to \(A_{\rm off}\) may be stale.
Hence \(\gamma^{(j)}\) can be written as
\[
    \gamma^{(j)}=\bigl(d^{(j)}\bigr)\trans\left[A_d e^{(j)}+A_{\rm off}\bigl(e^{(j)}+z^{(j)}\bigr)\right].
\]
Using \(A=A_d+A_{\rm off}\), this becomes
\begin{equation} \label{eq:assembled-gamma}
    \gamma^{(j)}=\bigl(d^{(j)}\bigr)\trans g^{(j)}, \qquad g^{(j)}:=Ae^{(j)}+A_{\rm off}z^{(j)}.
\end{equation}
Both \(e^{(j)}\) and \(z^{(j)}\), and therefore \(g^{(j)}\), depend only on the communication schedule and on the update directions before iteration~\(j\).
Hence \(g^{(j)}\) is independent of the current direction \(d^{(j)}\).
Since \(d^{(j)}\) is uniformly distributed over the coordinate directions, we have
\[
    \expect{d^{(j)}\bigl(d^{(j)}\bigr)\trans}=\frac1n I.
\]
Using~\eqref{eq:assembled-gamma}, we therefore have
\begin{equation} \label{eq:gamma-second-moment}
    \expect{(\gamma^{(j)})^2}
    =\expect{\bigl(g^{(j)}\bigr)\trans d^{(j)}\bigl(d^{(j)}\bigr)\trans g^{(j)}}
    =\frac1n\expect{\norm{g^{(j)}}^2}.
\end{equation}

Before proceeding with the analysis, we first present a triangle inequality in expectation that will be used.

\begin{lemma} \label{lem:triangle-ineq-new}
    For random vectors \(x,y\in\mathbb R^n\),
    \begin{equation} \label{eq:triangle-ineq-new-Anorm}
        \sqrt{\expect{\normA{x}^2}}-\sqrt{\expect{\normA{y}^2}}
        \leq \sqrt{\expect{\normA{x+y}^2}}
        \leq\sqrt{\expect{\normA{x}^2}}+\sqrt{\expect{\normA{y}^2}}.
    \end{equation}
    and
    \begin{equation} \label{eq:triangle-ineq-new}
        \sqrt{\expect{\norm{x}^2}}-\sqrt{\expect{\norm{y}^2}}
        \leq \sqrt{\expect{\norm{x+y}^2}}
        \leq\sqrt{\expect{\norm{x}^2}}+\sqrt{\expect{\norm{y}^2}}.
    \end{equation}
\end{lemma}

\begin{proof}
    The mapping \(\|z\|_{L^2(A)}:=\sqrt{\expect{\normA{z}^2}}\) is a norm.
    Hence the triangle inequality gives~\eqref{eq:triangle-ineq-new-Anorm}.
    We can similarly prove~\eqref{eq:triangle-ineq-new}.
\end{proof}

Then we are ready to show the main result of this step.

\begin{lemma} \label{lem:stale-energy-new}
    For every \(j\geq0\),
    \begin{equation} \label{eq:stale-energy-one-new}
        \frac1n\expect{\norm{A_{\rm off}z^{(j)}}^2}
        \leq\beta^2\rho\bigl(1+\sqrt{\rho\tau}\bigr)^2\sum_{k\in\tilde{\Phi}^{-}(j)}\expect{(\gamma^{(k)})^2}.
    \end{equation}
    Furthermore, for every \(j\geq0\) and integer \(l_0\geq1\),
    \begin{equation} \label{eq:stale-energy-block-new}
        \frac1n\sum_{i=j+\tau}^{j+\tau+l_0-1}\expect{\norm{A_{\rm off}z^{(i)}}^2}\leq\beta^2\omega_\tau^2\sum_{k=j}^{j+\tau+l_0-1}\expect{(\gamma^{(k)})^2}.
    \end{equation}
\end{lemma}

\begin{proof}
    For this proof, let \(\mathcal F^{(k)}\) denote the information determined by \(\mathcal S\), \(d^{(0)}\),~\(\dotsc\), \(d^{(k-1)}\).
    Then \(g^{(k)}\) is determined by \(\mathcal F^{(k)}\), while \(d^{(k)}\) is independent of \(\mathcal F^{(k)}\) and uniformly distributed over the coordinate directions.
    Define, only for this proof,
    \begin{equation} \label{eq:xi-definition}
        \xi^{(k)}:=d^{(k)}\gamma^{(k)}-\frac1n g^{(k)}.
    \end{equation}
    By~\eqref{eq:assembled-gamma},
    \[
        d^{(k)}\gamma^{(k)}=d^{(k)}\bigl(d^{(k)}\bigr)\trans g^{(k)}.
    \]
    Therefore, noticing that \(g^{(k)}\) is determined by \(\mathcal F^{(k)}\), we have
    \begin{align*}
        \expect{d^{(k)}\gamma^{(k)}\mid\mathcal F^{(k)}}
        &=\expect{d^{(k)}\bigl(d^{(k)}\bigr)\trans\mid\mathcal F^{(k)}}g^{(k)}
        =\frac1n g^{(k)}.
    \end{align*}
    Subtracting \(g^{(k)}/n\) from both sides gives
    \begin{equation} \label{eq:xi-md}
        \expect{\xi^{(k)}\mid\mathcal F^{(k)}}=0.
    \end{equation}
    Rearranging~\eqref{eq:xi-definition} yields
    \[
        d^{(k)}\gamma^{(k)}=\frac1n g^{(k)}+\xi^{(k)}.
    \]
    Substituting this equation into~\eqref{eq:z-missing-sum} gives
    \begin{equation}
        z^{(j)}=\beta\left(\frac1n\sum_{k\in\tilde{\Phi}^{-}(j)}g^{(k)}+\sum_{k\in\tilde{\Phi}^{-}(j)}\xi^{(k)}\right).
    \end{equation}
    which further gives, together with Lemma~\ref{lem:triangle-ineq-new},
    \begin{equation} \label{eq:z-drift-noise-new}
    \begin{split}
        &\left(\frac1n\expect{\norm{A_{\rm off}z^{(j)}}^2}\right)^{1/2} \\
        &\quad\leq\beta\left(\frac1n\expect{\norm{\frac1nA_{\rm off}\sum_{k\in\tilde{\Phi}^{-}(j)}g^{(k)}}^2}\right)^{1/2}
        + \beta\left(\frac1n\expect{\norm{A_{\rm off}\sum_{k\in\tilde{\Phi}^{-}(j)}\xi^{(k)}}^2}\right)^{1/2}.
    \end{split}
    \end{equation}
    We next estimate the two terms in~\eqref{eq:z-drift-noise-new} separately.

    We first consider the second terms involved in~\eqref{eq:z-drift-noise-new}.
    Since \(A\) is symmetric positive definite and has unit diagonal, we have, for \(r\neq t\), \(1-A_{r,t}^2>0\) and therefore \(\abs{A_{r,t}}<1\).
    Consequently, \(\abs{(A_{\rm off})_{r,t}}\leq1\) for all \(r,t\).
    For all \(I_t\),
    \begin{equation} \label{eq:Aoff-column-bound-new}
        \norm{A_{\rm off}I_t}^2
        =\sum_{r=1}^n\abs{(A_{\rm off})_{r,t}}^2 
        \leq\sum_{r=1}^n\abs{(A_{\rm off})_{r,t}}
        \leq n\rho.
    \end{equation}
    Conditioned on \(\mathcal F^{(k)}\), the vector \(g^{(k)}\) is fixed and \(d^{(k)}=I_t\) with probability \(1/n\).
    If \(d^{(k)}=I_t\), then~\eqref{eq:assembled-gamma} gives \(\gamma^{(k)}=(g^{(k)})_t\), and therefore
    \[
        \xi^{(k)}=I_t(g^{(k)})_t-\frac1n g^{(k)}.
    \]
    It follows that
    \begin{equation} \label{eq:xi-Aoff-bound-new}
    \begin{split}
        \expect{\norm{A_{\rm off}\xi^{(k)}}^2\mid\mathcal F^{(k)}}
        &=\frac1n\sum_{t=1}^n\left\|(g^{(k)})_tA_{\rm off}I_t-\frac1nA_{\rm off}g^{(k)}\right\|^2 \\
        &=\frac1n\sum_{t=1}^n(g^{(k)})_t^2\norm{A_{\rm off}I_t}^2-\frac1{n^2}\norm{A_{\rm off}g^{(k)}}^2 \\
        &\leq\frac1n\sum_{t=1}^n(g^{(k)})_t^2\norm{A_{\rm off}I_t}^2 \\
        &\leq\rho\norm{g^{(k)}}^2.
    \end{split}
    \end{equation}
    Taking expectation of~\eqref{eq:xi-Aoff-bound-new} and using~\eqref{eq:gamma-second-moment} gives
    \begin{equation} \label{eq:xi-Aoff-bound-uncond-new}
        \frac1n\expect{\norm{A_{\rm off}\xi^{(k)}}^2}\leq\rho\expect{(\gamma^{(k)})^2}.
    \end{equation}

    If \(k<r\), then \(\xi^{(k)}\) is determined by \(\mathcal F^{(r)}\).
    Using the tower property of conditional expectation and~\eqref{eq:xi-md},
    \begin{equation*}
    \begin{split}
        \expect{\bigl(\xi^{(k)}\bigr)\trans A_{\rm off}^2\xi^{(r)}}
        &=\expect{\expect{\bigl(\xi^{(k)}\bigr)\trans A_{\rm off}^2\xi^{(r)}\mid\mathcal F^{(r)}}} \\
        &=\expect{\bigl(\xi^{(k)}\bigr)\trans A_{\rm off}^2\expect{\xi^{(r)}\mid\mathcal F^{(r)}}} \\
        &=0.
    \end{split}
    \end{equation*}
    Together with~\eqref{eq:xi-Aoff-bound-uncond-new}, we have
    \begin{equation} \label{eq:martingale-sum-new}
    \begin{split}
        \frac1n\expect{\norm{A_{\rm off}\sum_{k\in\tilde{\Phi}^{-}(j)}\xi^{(k)}}^2}
        &=\frac1n\sum_{k\in\tilde{\Phi}^{-}(j)}\expect{\norm{A_{\rm off}\xi^{(k)}}^2} \\
        &\leq\rho\sum_{k\in\tilde{\Phi}^{-}(j)}\expect{(\gamma^{(k)})^2}.
    \end{split}
    \end{equation}

    We next estimate the first term in~\eqref{eq:z-drift-noise-new}.
    Since \(A_{\rm off}\) is symmetric,
    \[
        \norm{A_{\rm off}}\leq\sqrt{\norm{A_{\rm off}}_1\norm{A_{\rm off}}_\infty}=\norm{A_{\rm off}}_1\leq n\rho.
    \]
    Also, by the Cauchy--Schwarz inequality and \(\abs{\tilde{\Phi}^{-}(j)}\leq\tau\),
    \begin{align*}
        \norm{\sum_{k\in\tilde{\Phi}^{-}(j)}g^{(k)}}^2
        &\leq\tau\sum_{k\in\tilde{\Phi}^{-}(j)}\norm{g^{(k)}}^2.
    \end{align*}
    This further implies that
    \begin{equation} \label{eq:drift-sum-new}
    \begin{split}
        \frac1n\expect{\norm{\frac1nA_{\rm off}\sum_{k\in\tilde{\Phi}^{-}(j)}g^{(k)}}^2}
        &\leq\frac{\norm{A_{\rm off}}^2}{n^3}\expect{\norm{\sum_{k\in\tilde{\Phi}^{-}(j)}g^{(k)}}^2} \\
        &\leq\frac{n^2\rho^2}{n^3}\tau\sum_{k\in\tilde{\Phi}^{-}(j)}\expect{\norm{g^{(k)}}^2} \\
        &=\rho^2\tau\sum_{k\in\tilde{\Phi}^{-}(j)}\expect{(\gamma^{(k)})^2},
    \end{split}
    \end{equation}
    where the last equality follows from~\eqref{eq:gamma-second-moment}.

    Let
    \[
        R_j:=\sum_{k\in\tilde{\Phi}^{-}(j)}\expect{(\gamma^{(k)})^2}.
    \]
    Using~\eqref{eq:drift-sum-new} and~\eqref{eq:martingale-sum-new} on the right-hand side of~\eqref{eq:z-drift-noise-new}, we obtain
    \begin{equation*}
        \left(\frac1n\expect{\norm{A_{\rm off}z^{(j)}}^2}\right)^{1/2}
        \leq\beta\sqrt{\rho}\bigl(1+\sqrt{\rho\tau}\bigr)\sqrt{\sum_{k\in\tilde{\Phi}^{-}(j)}\expect{(\gamma^{(k)})^2}},
    \end{equation*}
    which proves~\eqref{eq:stale-energy-one-new} by squaring both sides.

    It remains to prove~\eqref{eq:stale-energy-block-new}.
    Summing~\eqref{eq:stale-energy-one-new} over \(i=j+\tau,\dotsc,j+\tau+l_0-1\) gives
    \begin{align*}
        \frac1n\sum_{i=j+\tau}^{j+\tau+l_0-1}\expect{\norm{A_{\rm off}z^{(i)}}^2}
        &\leq\beta^2\rho\bigl(1+\sqrt{\rho\tau}\bigr)^2\sum_{i=j+\tau}^{j+\tau+l_0-1}\sum_{k\in\tilde{\Phi}^{-}(i)}\expect{(\gamma^{(k)})^2}.
    \end{align*}
    By the bounded-delay assumption, \(k\in\tilde{\Phi}^{-}(i)\) implies \(i-\tau\leq k\leq i-1\).
    Since \(i\geq j+\tau\), this implies \(k\geq j\).
    Since \(i\leq j+\tau+l_0-1\), it also implies \(k\leq j+\tau+l_0-2\).
    Moreover, for a fixed \(k\), the relation \(k\in\tilde{\Phi}^{-}(i)\) can hold only for \(i=k+1,\dotsc,k+\tau\), so each \(k\) is counted at most \(\tau\) times.
    Therefore,
    \begin{equation} \label{eq:missing-count-block-new}
        \sum_{i=j+\tau}^{j+\tau+l_0-1}\sum_{k\in\tilde{\Phi}^{-}(i)}\expect{(\gamma^{(k)})^2}\leq\tau\sum_{k=j}^{j+\tau+l_0-1}\expect{(\gamma^{(k)})^2}.
    \end{equation}
    Combining the last two inequalities gives~\eqref{eq:stale-energy-block-new}.
\end{proof}

\paragraph{Step 2: Constructing a decreasing modified error.}
We first derive a one-step estimate that separates the progress of the current update from the effect of the missing updates.

\begin{lemma} \label{lem:one-step-new}
    For \(\rho>0\), let
    \[
        \eta_\tau:=\beta\omega_\tau, \qquad \delta_\tau:=\frac{\beta^2\omega_\tau}{\tau}.
    \]
    Then
    \begin{equation} \label{eq:one-step-new}
        \expectabbr{j+1}\leq\expectabbr{j}-\beta\bigl(2-\beta-\eta_\tau\bigr)\expect{(\gamma^{(j)})^2}+\delta_\tau\sum_{k\in\tilde{\Phi}^{-}(j)}\expect{(\gamma^{(k)})^2}.
    \end{equation}
    If \(\rho=0\), the same bound holds with \(\eta_\tau=\delta_\tau=0\).
\end{lemma}

\begin{proof}
    From the iteration formula,
    \[
        x^{(j+1)}-x^\star=x^{(j)}-x^\star+\beta d^{(j)}\gamma^{(j)}.
    \]
    Expanding the squared \(A\)-norm gives
    \begin{equation} \label{eq:lem-proof:normxj1-xs}
    \begin{split}
        &\normA{x^{(j+1)}-x^\star}^2 \\
        &\quad=\normA{x^{(j)}-x^\star}^2+2\beta\gamma^{(j)}\bigl(d^{(j)}\bigr)\trans A\bigl(x^{(j)}-x^\star\bigr)
        +\beta^2(\gamma^{(j)})^2\bigl(d^{(j)}\bigr)\trans A d^{(j)} \\
        &\quad=\normA{x^{(j)}-x^\star}^2+2\beta\gamma^{(j)}\bigl(d^{(j)}\bigr)\trans A\bigl(x^{(j)}-x^\star\bigr)
        +\beta^2(\gamma^{(j)})^2,
    \end{split}
    \end{equation}
    where the last equality holds since \(A\) has unit diagonal and \(d^{(j)}\) is a coordinate vector.
    By \(e^{(j)}=x^\star-x^{(j)}\),~\eqref{eq:assembled-gamma} implies
    \begin{equation} \label{eq:lem-proof:dAxjxs}
    \begin{split}
        \bigl(d^{(j)}\bigr)\trans A\bigl(x^{(j)}-x^\star\bigr)
        &=-\bigl(d^{(j)}\bigr)\trans Ae^{(j)}
        =-\gamma^{(j)}+\bigl(d^{(j)}\bigr)\trans A_{\rm off}z^{(j)}.
    \end{split}
    \end{equation}
    Substituting~\eqref{eq:lem-proof:dAxjxs} for \(\bigl(d^{(j)}\bigr)\trans A\bigl(x^{(j)}-x^\star\bigr)\) involved in~\eqref{eq:lem-proof:normxj1-xs}, we have
    \begin{equation} \label{eq:energy-identity-new}
        \normA{x^{(j+1)}-x^\star}^2=\normA{x^{(j)}-x^\star}^2-\beta(2-\beta)(\gamma^{(j)})^2+2\beta\gamma^{(j)}\bigl(d^{(j)}\bigr)\trans A_{\rm off}z^{(j)}.
    \end{equation}

    Next, we evaluate the expectation of the last term in~\eqref{eq:energy-identity-new}.
    Using~\eqref{eq:assembled-gamma}, we obtain
    \[
        \gamma^{(j)}\bigl(d^{(j)}\bigr)\trans A_{\rm off}z^{(j)}
        =\bigl(g^{(j)}\bigr)\trans d^{(j)}\bigl(d^{(j)}\bigr)\trans A_{\rm off}z^{(j)}.
    \]
    Noticing that \(g^{(j)}\) and \(z^{(j)}\) do not depend on \(d^{(j)}\), we can take the expectation of \(\gamma^{(j)}\bigl(d^{(j)}\bigr)\trans A_{\rm off}z^{(j)}\) to obtain
    \begin{equation} \label{eq:cross-conditional-new}
    \begin{split}
        \expect{\gamma^{(j)}\bigl(d^{(j)}\bigr)\trans A_{\rm off}z^{(j)}}
        &=\expect{\bigl(g^{(j)}\bigr)\trans\expect{d^{(j)}\bigl(d^{(j)}\bigr)\trans}A_{\rm off}z^{(j)}}
        =\frac1n\expect{\bigl(g^{(j)}\bigr)\trans A_{\rm off}z^{(j)}}.
    \end{split}
    \end{equation}
    Then we will handle \(\expect{\bigl(g^{(j)}\bigr)\trans A_{\rm off}z^{(j)}}/n\).
    For any \(\eta>0\), Young's inequality gives
    \[
        2\bigl(g^{(j)}\bigr)\trans A_{\rm off}z^{(j)}\leq\eta\norm{g^{(j)}}^2+\frac1{\eta}\norm{A_{\rm off}z^{(j)}}^2.
    \]
    Dividing by \(n\), taking expectations, and using~\eqref{eq:gamma-second-moment} yields
    \begin{equation} \label{eq:cross-young-new}
        \frac{2}{n}\expect{\bigl(g^{(j)}\bigr)\trans A_{\rm off}z^{(j)}}\leq\eta\expect{(\gamma^{(j)})^2}+\frac{1}{\eta n}\expect{\norm{A_{\rm off}z^{(j)}}^2}.
    \end{equation}
    
    By taking expectations in~\eqref{eq:energy-identity-new} and using~\eqref{eq:cross-conditional-new},~\eqref{eq:cross-young-new}, we have
    \begin{align*}
        \expectabbr{j+1}
        &\leq\expectabbr{j}-\beta(2-\beta-\eta)\expect{(\gamma^{(j)})^2}+\frac{\beta}{\eta n}\expect{\norm{A_{\rm off}z^{(j)}}^2},
    \end{align*}
    which further implies that, using~\eqref{eq:stale-energy-one-new},
    \begin{equation} \label{eq:one-step-eta-new}
        \expectabbr{j+1}\leq\expectabbr{j}-\beta(2-\beta-\eta)\expect{(\gamma^{(j)})^2}+\frac{\beta^3\rho(1+\sqrt{\rho\tau})^2}{\eta}\sum_{k\in\tilde{\Phi}^{-}(j)}\expect{(\gamma^{(k)})^2}.
    \end{equation}
    By choosing \(\eta=\eta_\tau\) in~\eqref{eq:one-step-eta-new}, we conclude the proof.
\end{proof}

\paragraph{Step 3: Establishing progress over a block of iterations.}
To prove Theorem~\ref{thm:improved}, we realize that one of the main difficulties arises from the fact that the expected error \(\expectabbr{j}\) is not necessarily monotone under asynchronous updates.
This lack of monotonicity makes it difficult to directly track the progress of the algorithm over multiple iterations and leads to rather conservative estimates when errors at different iterations need to be compared.
To address this issue, we augment \(\expectabbr{j}\) with an additional term that captures the effect of pending asynchronous updates, thereby constructing \(\Etemp{j}\).
As shown below, this additional term is designed so that \(\Etemp{j}\) is monotonically decreasing, which provides a more suitable quantity for the refined convergence analysis.
More precisely, we define
\begin{equation} \label{eq:def-Etemp-new}
    \Etemp{j}:=\expectabbr{j}+\delta_\tau\sum_{k=0}^{j-1}m^{(j)}(k)\expect{(\gamma^{(k)})^2}.
\end{equation}
where \(m^{(j)}(k):=\abs{\{i\geq j: k\in \tilde{\Phi}^{-}(i)\}}\) satisfies
\begin{equation} \label{eq:mjk-properties-new}
    0\leq m^{(j)}(k)\leq\tau, \qquad m^{(j+\tau)}(k)=0\quad\text{with}\quad k<j.
\end{equation}
In Property~\ref{property:Etemp-new}, we present the basic properties of \(\Etemp{j}\).

\begin{property} \label{property:Etemp-new}
    Under the assumptions of Theorem~\ref{thm:improved}, then
    \begin{equation} \label{eq:property-Etemp-ge-E-new}
        \expectabbr{j}\leq\Etemp{j},
    \end{equation}
    \begin{equation} \label{eq:property-monodecrease-new}
        \Etemp{j+1}\leq\Etemp{j}-a_\tau \expect{(\gamma^{(j)})^2}\leq\Etemp{j},
    \end{equation}
    and, for every integer \(l\geq1\),
    \begin{equation} \label{eq:property-monodecrease-l-new}
        \Etemp{j+l}\leq\Etemp{j}-a_\tau\sum_{k=j}^{j+l-1}\expect{(\gamma^{(k)})^2}.
    \end{equation}
    Moreover,
    \begin{equation} \label{eq:Etemp-flush-new}
        \Etemp{j+\tau}\leq\expectabbr{j+\tau}+\beta^2\omega_\tau\sum_{k=j}^{j+\tau-1}\expect{(\gamma^{(k)})^2}.
    \end{equation}
\end{property}

\begin{proof}
    From the definition~\eqref{eq:def-Etemp-new} of \(\Etemp{j}\), it is straightforward to see that \(\expectabbr{j} \leq\Etemp{j}\).

    We will then prove~\eqref{eq:property-monodecrease-l-new}.
    For every \(k<j\), the set of future reads counted by \(m^{(j+1)}(k)\) is the set counted by \(m^{(j)}(k)\) with iteration \(j\) removed if update \(k\) is missing from that read.
    Therefore,
    \begin{equation} \label{eq:m-difference-new}
        m^{(j+1)}(k)-m^{(j)}(k)=-\mathbf1_{\{k\in \tilde{\Phi}^{-}(j)\}}
    \end{equation}
    Using~\eqref{eq:m-difference-new}, we have
    \begin{equation} \label{eq:debt-difference-new}
    \begin{split}
        &\delta_\tau\sum_{k=0}^{j}m^{(j+1)}(k)\expect{(\gamma^{(k)})^2}-\delta_\tau\sum_{k=0}^{j-1}m^{(j)}(k)\expect{(\gamma^{(k)})^2} \\
        &=\delta_\tau\sum_{k=0}^{j-1}\bigl(m^{(j+1)}(k)-m^{(j)}(k)\bigr)\expect{(\gamma^{(k)})^2}+\delta_\tau m^{(j+1)}(j)\expect{(\gamma^{(j)})^2} \\
        &=-\delta_\tau\sum_{k\in \tilde{\Phi}^{-}(j)}\expect{(\gamma^{(k)})^2}+\delta_\tau m^{(j+1)}(j)\expect{(\gamma^{(j)})^2}.
    \end{split}
    \end{equation}

    Combining~\eqref{eq:one-step-new} with~\eqref{eq:debt-difference-new} and~\eqref{eq:def-Etemp-new} and noticing that \(m^{(j+1)}(j)\leq\tau\), we obtain
    \begin{align*}
        \Etemp{j+1}-\Etemp{j}
        &\leq-\beta(2-\beta-\eta_\tau)\expect{(\gamma^{(j)})^2}+\delta_\tau m^{(j+1)}(j)\expect{(\gamma^{(j)})^2}\\
        &\leq-\left[\beta(2-\beta-\eta_\tau)-\delta_\tau\tau\right]\expect{(\gamma^{(j)})^2},
    \end{align*}
    which proves~\eqref{eq:property-monodecrease-new} by the definitions of \(\eta_\tau\) and \(\delta_\tau\).
    By summing, we can further prove~\eqref{eq:property-monodecrease-l-new}.

    It remains to bound \(\Etemp{j+\tau}\).
    By~\eqref{eq:def-Etemp-new} and~\eqref{eq:mjk-properties-new}, \(\Etemp{j+\tau}\) can be bounded by
    \begin{align*}
        \Etemp{j+\tau}
        &=\expectabbr{j+\tau}+\delta_\tau\sum_{k=j}^{j+\tau-1}m^{(j+\tau)}(k)\expect{(\gamma^{(k)})^2}\\
        &\leq\expectabbr{j+\tau}+\delta_\tau\tau\sum_{k=j}^{j+\tau-1}\expect{(\gamma^{(k)})^2},
    \end{align*}
    which proves~\eqref{eq:Etemp-flush-new} by using the definition of \(\delta_\tau\).
\end{proof}

Using~\eqref{eq:property-monodecrease-l-new}, to derive a block contraction for \(\Etemp{j}\), it suffices to bound the second term, i.e., \(\sum_{k=j}^{j+l-1}\expect{(\gamma^{(k)})^2}\), which is addressed in Lemma~\ref{lem:block-progress-new}.

\begin{lemma} \label{lem:block-progress-new}
    Let \(l_0\geq1\), \(l=\tau+l_0\), and \(c_{l_0,\tau}\) be defined by~\eqref{eq:def-cl0tau-new}.
    Then, for every \(j\geq0\),
    \begin{equation} \label{eq:block-progress-new}
        \sum_{k=j}^{j+l-1}\expect{(\gamma^{(k)})^2}\geq c_{l_0,\tau}\expectabbr{j}.
    \end{equation}
\end{lemma}

\begin{proof}
    For every \(i\in\{j+\tau,\dotsc,j+\tau+l_0-1\}\),~\eqref{eq:gamma-second-moment} and~\eqref{eq:assembled-gamma} give
    \begin{equation} \label{eq:block-gamma-new}
        \expect{(\gamma^{(i)})^2}=\frac1n\expect{\norm{Ae^{(i)}+A_{\rm off}z^{(i)}}^2}.
    \end{equation}
    By summing on \(i\), we have
    \begin{equation} \label{eq:block-gamma-sum}
        \sum_{i=j+\tau}^{j+l-1}\expect{(\gamma^{(i)})^2}=\frac1n\sum_{i=j+\tau}^{j+l-1}\expect{\norm{Ae^{(i)}+A_{\rm off}z^{(i)}}^2}.
    \end{equation}
    Similarly to Lemma~\ref{lem:triangle-ineq-new}, we can prove the triangle inequality for \(\sum_{i=j+\tau}^{j+l-1}\expect{(\gamma^{(i)})^2}\), which implies
    \begin{equation*}
        \left(\sum_{i=j+\tau}^{j+l-1}\expect{(\gamma^{(i)})^2}\right)^{1/2}\geq\left(\frac1n\sum_{i=j+\tau}^{j+l-1}\expect{\norm{Ae^{(i)}}^2}\right)^{1/2}-\left(\frac1n\sum_{i=j+\tau}^{j+l-1}\expect{\norm{A_{\rm off}z^{(i)}}^2}\right)^{1/2}.
    \end{equation*}
    Furthermore, noticing \(j+\tau\geq j\), it follows that
    \begin{equation} \label{eq:block-minkowski-new}
        \left(\sum_{i=j}^{j+l-1}\expect{(\gamma^{(i)})^2}\right)^{1/2}\geq\left(\frac1n\sum_{i=j+\tau}^{j+l-1}\expect{\norm{Ae^{(i)}}^2}\right)^{1/2}-\left(\frac1n\sum_{i=j+\tau}^{j+l-1}\expect{\norm{A_{\rm off}z^{(i)}}^2}\right)^{1/2}.
    \end{equation}
    Next, we bound the two terms on the right-hand side of~\eqref{eq:block-minkowski-new} separately.

    We first bound the first term on the right-hand side of~\eqref{eq:block-minkowski-new}.
    Notice that
    \[
        \norm{Ae^{(i)}}^2=\bigl(e^{(i)}\bigr)\trans A^2e^{(i)}\geq\lambda_{\min}(A)\bigl(e^{(i)}\bigr)\trans Ae^{(i)}.
    \]
    By dividing by \(n\) and recalling \(\mu=\lambda_{\min}(A)/n\), we obtain
    \begin{equation} \label{eq:Ae-lower-new}
        \frac1n\norm{Ae^{(i)}}^2\geq\mu\normA{e^{(i)}}^2.
    \end{equation}
    We then move on to handling \(\normA{e^{(i)}}^2\).
    Applying Lemma~\ref{lem:triangle-ineq-new} to
    \[
        -e^{(i)} = x^{(i)}-x^\star=(x^{(j)}-x^\star)+(x^{(i)}-x^{(j)}),
    \]
    we obtain
    \begin{equation} \label{eq:lem-proof:Eisqrt}
    \begin{split}
        \sqrt{\expectabbr{i}}
        &\geq\sqrt{\expectabbr{j}}-\sqrt{\expect{\normA{x^{(i)}-x^{(j)}}^2}}.
    \end{split}
    \end{equation}
    We now consider \(\expect{\normA{x^{(i)}-x^{(j)}}^2}\).
    Since
    \[
        x^{(i)}-x^{(j)}=\beta\sum_{k=j}^{i-1}d^{(k)}\gamma^{(k)},
    \]
    the Cauchy--Schwarz inequality for the \(A\)-norm gives
    \begin{equation} \label{eq:lem-proof:xi-xj}
        \normA{x^{(i)}-x^{(j)}}^2
        \leq\beta^2(i-j)\sum_{k=j}^{i-1}\normA{d^{(k)}\gamma^{(k)}}^2.
    \end{equation}
    Notice that, by \(A\) having unit diagonal,
    \[
        \normA{d^{(k)}\gamma^{(k)}}^2=(\gamma^{(k)})^2\bigl(d^{(k)}\bigr)\trans A d^{(k)}=(\gamma^{(k)})^2.
    \]
    By taking expectations of~\eqref{eq:lem-proof:xi-xj} and \(i-j\leq l-1\), we have
    \begin{equation} \label{eq:lem-proof:Exi-xj}
        \expect{\normA{x^{(i)}-x^{(j)}}^2}
        \leq\beta^2(l-1)\sum_{k=j}^{i-1}\expect{(\gamma^{(k)})^2}
        \leq\beta^2(l-1)\sum_{k=j}^{j+l-1}\expect{(\gamma^{(k)})^2}.
    \end{equation}
    Together with~\eqref{eq:lem-proof:Eisqrt}, we further obtain
    \begin{equation} \label{eq:Ei-movement-new}
    \begin{split}
        \sqrt{\expectabbr{i}}
        &\geq\sqrt{\expectabbr{j}}-\beta\sqrt{(l-1)\sum_{k=j}^{j+l-1}\expect{(\gamma^{(k)})^2}}.
    \end{split}
    \end{equation}

    Then we will bound the second term involved in~\eqref{eq:Ei-movement-new} by distinguishing two cases.
    Suppose first that
    \[
        \beta\sqrt{(l-1)\sum_{k=j}^{j+l-1}\expect{(\gamma^{(k)})^2}}\geq\sqrt{\expectabbr{j}}.
    \]
    Squaring and dividing by \(\beta^2T\) gives
    \[
        \sum_{k=j}^{j+l-1}\expect{(\gamma^{(k)})^2}\geq\frac{\expectabbr{j}}{\beta^2(l-1)}
        \geq c_{l_0,\tau}\expectabbr{j},
    \]
    where the last inequality is due to \(c_{l_0,\tau}\leq 1/(\beta^2(l-1))\) by the definition of \(c_{l_0,\tau}\).
    This proves~\eqref{eq:block-progress-new} in this case.
    We now consider the case
    \[
        \beta\sqrt{(l-1)\sum_{k=j}^{j+l-1}\expect{(\gamma^{(k)})^2}}<\sqrt{\expectabbr{j}}.
    \]
    Then the right-hand side of~\eqref{eq:Ei-movement-new} is positive.
    For every \(i\in\{j+\tau,\dotsc,j+l-1\}\),~\eqref{eq:Ae-lower-new} and~\eqref{eq:Ei-movement-new} imply
    \begin{align*}
        \frac1n\expect{\norm{Ae^{(i)}}^2}
        &\geq\mu\expectabbr{i}
        \geq\mu\left(\sqrt{\expectabbr{j}}-\beta\sqrt{(l-1)\sum_{k=j}^{j+l-1}\expect{(\gamma^{(k)})^2}}\right)^2.
    \end{align*}
    Summing this inequality over the \(l_0\) indices in \(\{j+\tau,\dotsc,j+l-1\}\) gives
    \[
        \frac1n\sum_{k=j+\tau}^{j+l-1}\expect{\norm{Ae^{(i)}}^2}\geq\mu l_0\left(\sqrt{\expectabbr{j}}-\beta\sqrt{(l-1)\sum_{k=j}^{j+l-1}\expect{(\gamma^{(k)})^2}}\right)^2,
    \]
    which further yields, by taking square roots of this inequality,
    \begin{equation} \label{eq:block-current-lower-new}
        \left(\frac1n\sum_{k=j+\tau}^{j+l-1}\expect{\norm{Ae^{(i)}}^2}\right)^{1/2}\geq\sqrt{\mu l_0}\left(\sqrt{\expectabbr{j}}-\beta\sqrt{(l-1)\sum_{k=j}^{j+l-1}\expect{(\gamma^{(k)})^2}}\right).
    \end{equation}

    It remains to bound the second term on the right-hand side of~\eqref{eq:block-minkowski-new}.
    By~\eqref{eq:stale-energy-block-new} in Lemma~\ref{lem:stale-energy-new},
    \begin{equation} \label{eq:block-stale-upper-new}
        \left(\frac1n\sum_{k=j+\tau}^{j+l-1}\expect{\norm{A_{\rm off}z^{(i)}}^2}\right)^{1/2}\leq\beta\omega_\tau\sqrt{\sum_{k=j}^{j+l-1}\expect{(\gamma^{(k)})^2}}.
    \end{equation}
    Only for this proof, let
    \[
        S^{(j)}:=\sum_{k=j}^{j+l-1}\expect{(\gamma^{(k)})^2}.
    \]
    Substituting~\eqref{eq:block-current-lower-new} and~\eqref{eq:block-stale-upper-new} into~\eqref{eq:block-minkowski-new} gives
    \begin{align*}
        \sqrt{S^{(j)}}
        &\geq\sqrt{\mu l_0}\left(\sqrt{\expectabbr{j}}-\beta\sqrt{(l-1)S^{(j)}}\right)-\beta\omega_\tau\sqrt{S^{(j)}},
    \end{align*}
    which implies that
    \begin{align*}
        S^{(j)}
        &\geq\frac{\mu l_0}{\bigl(1+\beta\sqrt{\mu l_0T}+\beta\omega_\tau\bigr)^2}\expectabbr{j}
        =c_{l_0,\tau}\expectabbr{j}.
    \end{align*}
    This proves~\eqref{eq:block-progress-new} in the second case and completes the proof.
\end{proof}

\paragraph{Step 4: Deriving a block contraction.}
We now combine the monotonicity of \(\Etemp{j}\) with Lemma~\ref{lem:block-progress-new}.

\begin{lemma} \label{lem:cases-new}
    Let \(l_0\geq1\), \(l=\tau+l_0\), \(0<\theta<1\), and \(\varepsilon>0\).
    \begin{enumerate}
        \item If \(\expectabbr{j}\geq\theta\Etemp{j}\), then
        \begin{equation} \label{eq:case1-new}
            \Etemp{j+l}\leq\bigl(1-a_\tau c_{l_0,\tau}\theta\bigr)\Etemp{j}.
        \end{equation}
        \item If \(\expectabbr{j}<\theta\Etemp{j}\), then
        \begin{equation} \label{eq:case2-new}
            \Etemp{j+l}\leq\max\left\{1-a_\tau\varepsilon,\ \theta+\beta^2(\tau+\omega_\tau)\varepsilon+2\beta\sqrt{\theta\tau\varepsilon}\right\}\Etemp{j}.
        \end{equation}
    \end{enumerate}
\end{lemma}

\begin{proof}
    We first consider \(\expectabbr{j}\geq\theta\Etemp{j}\).
    By~\eqref{eq:property-monodecrease-l-new} with \(l=\tau+l_0\) and then using Lemma~\ref{lem:block-progress-new}, we have
    \begin{align*}
        \Etemp{j+l}
        &\leq\Etemp{j}-a_\tau c_{l_0,\tau}\expectabbr{j}
        =\bigl(1-a_\tau c_{l_0,\tau}\theta\bigr)\Etemp{j},
    \end{align*}
    which proves~\eqref{eq:case1-new}.

    We then consider \(\expectabbr{j}<\theta\Etemp{j}\) by distinguishing two cases according to the progress made during the first \(\tau\) iterations of the block.
    If
    \[
        \sum_{k=j}^{j+\tau-1}\expect{(\gamma^{(k)})^2}\geq\varepsilon\Etemp{j},
    \]
    then~\eqref{eq:property-monodecrease-l-new} together with~\eqref{eq:property-monodecrease-new} gives
    \begin{align*}
        \Etemp{j+l}\leq\Etemp{j+\tau}
        &\leq\Etemp{j}-a_\tau\sum_{k=j}^{j+\tau-1}\expect{(\gamma^{(k)})^2}
        \leq\bigl(1-a_\tau\varepsilon\bigr)\Etemp{j}.
    \end{align*}

    It remains to consider
    \[
        \sum_{k=j}^{j+\tau-1}\expect{(\gamma^{(k)})^2}<\varepsilon\Etemp{j}.
    \]
    From~\eqref{eq:lem-proof:Exi-xj} with \(i=j+\tau\), we have
    \begin{equation} \label{lem:proof:Exjtau-xj}
        \expect{\normA{x^{(j+\tau)}-x^{(j)}}^2}
        \leq\beta^2\tau\sum_{k=j}^{j+\tau-1}\expect{(\gamma^{(k)})^2}
        <\beta^2\tau\varepsilon\Etemp{j}.
    \end{equation}
    Taking expectations of 
    \[
        x^{(j+\tau)}-x^\star=(x^{(j)}-x^\star)+(x^{(j+\tau)}-x^{(j)})
    \]
    and then applying Lemma~\ref{lem:triangle-ineq-new}, we obtain
    \begin{align*}
        \sqrt{\expectabbr{j+\tau}}
        &\leq\sqrt{\expectabbr{j}}+\sqrt{\expect{\normA{x^{(j+\tau)}-x^{(j)}}^2}}\\
        &<\sqrt{\theta\Etemp{j}}+\beta\sqrt{\tau\varepsilon\Etemp{j}}\\
        &=\left(\sqrt\theta+\beta\sqrt{\tau\varepsilon}\right)\sqrt{\Etemp{j}},
    \end{align*}
    where the second inequality is derived from the assumption of this case and using~\eqref{lem:proof:Exjtau-xj}. 
    We further have
    \begin{equation} \label{eq:Ejplus-tau-small-new}
        \expectabbr{j+\tau}\leq\left(\sqrt\theta+\beta\sqrt{\tau\varepsilon}\right)^2\Etemp{j}.
    \end{equation}
    Together with~\eqref{eq:property-monodecrease-new} and~\eqref{eq:Etemp-flush-new}, we derive
    \begin{align*}
        \Etemp{j+l}\leq\Etemp{j+\tau}
        &\leq\left[\left(\sqrt\theta+\beta\sqrt{\tau\varepsilon}\right)^2+\beta^2\omega_\tau\varepsilon\right]\Etemp{j}
        =\left[\theta+\beta^2(\tau+\omega_\tau)\varepsilon+2\beta\sqrt{\theta\tau\varepsilon}\right]\Etemp{j}.
    \end{align*}
    Combining these two cases, we can prove~\eqref{eq:case2-new}.
\end{proof}

\paragraph{Step 5: Optimizing the contraction bound.}
Combining the two cases shown in Lemma~\ref{lem:cases-new}, we have
\begin{equation} \label{eq:minimax-main-new}
    \Etemp{j+l}\leq\max\left\{1-a_\tau c_{l_0,\tau}\theta,\, 1-a_\tau\varepsilon,\, \theta+\beta^2(\tau+\omega_\tau)\varepsilon+2\beta\sqrt{\theta\tau\varepsilon}\right\}\Etemp{j},
\end{equation}
which introduces auxiliary parameters, i.e., \(\theta\) and \(\varepsilon\), that determine the threshold between the different cases.
To minimize the worst of the contraction factors in~\eqref{eq:minimax-main-new}, for a fixed \(l_0\), we choose these parameters using a minimax optimization, that is,
\begin{equation} \label{eq:thm-proof:minimax}
    1-a_\tau c_{l_0,\tau}\theta
    = 1-a_\tau\varepsilon
    = \theta+\beta^2(\tau+\omega_\tau)\varepsilon+2\beta\sqrt{\theta\tau\varepsilon}.
\end{equation}
We are now ready to prove Theorem~\ref{thm:improved} by solving this minimax problem.

\begin{proof}[Proof of Theorem~\ref{thm:improved}]
Let \(c:=c_{l_0,\tau}\) for this paragraph.
For fixed \(l_0\), the minimax problem~\eqref{eq:thm-proof:minimax} is solved by balancing the three branches, i.e.,
\begin{align}
    1-a_\tau  c_{l_0,\tau}\theta^{\star} &=1-a_\tau\varepsilon^{\star}, \label{eq:thm-proof:factor12}\\
    1-a_\tau c_{l_0,\tau}\theta^{\star} &=\theta^{\star}+\beta^2(\tau+\omega_\tau)\varepsilon^{\star}+2\beta\sqrt{\theta^{\star}\tau\varepsilon^{\star}}. \label{eq:thm-proof:factor23}
\end{align}
From~\eqref{eq:thm-proof:factor12}, we have
\begin{equation} \label{eq:eps-theta-balance-new}
    \varepsilon^\star=c_{l_0,\tau}\theta^\star.
\end{equation}
It remains to consider~\eqref{eq:thm-proof:factor23}.
Substituting~\eqref{eq:eps-theta-balance-new} for \(\varepsilon^{\star}\) in~\eqref{eq:thm-proof:factor23}, we obtain
\[
    1-a_\tau c_{l_0,\tau}\theta^\star=\theta^\star\bigl(1+\beta^2(\tau+\omega_\tau)c+2\beta\sqrt{\tau c_{l_0,\tau}}\bigr),
\]
which gives
\begin{equation} \label{eq:theta-star-new}
    \theta^\star=\frac{1}{1+\bigl(a_\tau+\beta^2(\tau+\omega_\tau)\bigr)c_{l_0,\tau}+2\beta\sqrt{\tau c_{l_0,\tau}}}.
\end{equation}
Substituting~\eqref{eq:theta-star-new} for \(\theta=\theta^{\star}\) in~\eqref{eq:minimax-main-new}, we have
\begin{equation}
    \Etemp{j+l}\leq\left(1-\frac{a_\tau c_{l_0,\tau}}{1+\bigl(a_\tau+\beta^2(\tau+\omega_\tau)\bigr)c_{l_0,\tau}+2\beta\sqrt{\tau c_{l_0,\tau}}}\right)\Etemp{j}.
\end{equation}
    Let \(j+l=tl\).
    We then obtain
    \begin{equation}
        \Etemp{tl}\leq\left(1-\frac{a_\tau c_{l_0,\tau}}{1+\bigl(a_\tau+\beta^2(\tau+\omega_\tau)\bigr)c_{l_0,\tau}+2\beta\sqrt{\tau c_{l_0,\tau}}}\right)^t\Etemp{0},
    \end{equation}
    which concludes the proof by noticing \(\Etemp{0}=\expectabbr{0}\) and \(\expectabbr{tl}\leq\Etemp{tl}\).
\end{proof}

\begin{remark}[Consistency with the synchronous case] \label{rem:sync-consistency}
The convergence factor obtained from the minimax optimization above does not reduce to the standard synchronous bound by simply setting \(\tau=0\).
This discrepancy is an artifact of the two-case argument and the subsequent optimization of the auxiliary parameters, rather than an intrinsic loss of the analysis in the synchronous setting.
Indeed, the minimax optimization is designed to provide a uniform contraction over both cases by choosing \(0<\theta<1\) and balancing the corresponding contraction factors.
However, when \(\tau=0\), \(\tilde{\Phi}^{-}(j)=\emptyset\) and in \(\Etemp{j}=\expectabbr{j}\).
Hence, by taking \(\theta=1\), every iteration falls into the first case, whereas the second case is vacuous.
The auxiliary parameter \(\varepsilon\) and the associated minimax balancing are therefore unnecessary.
In this case, the one-step estimate in Lemma~\ref{lem:one-step-new} reduces directly to
\[
\expectabbr{j+1} = \expectabbr{j} - \beta(2-\beta)\expect{(\gamma^{(j)})^2}.
\]
Since
\[
\expect{(\gamma^{(j)})^2} = \frac{1}{n}\expect{\norm{Ae^{(j)}}^2} \geq \mu\expectabbr{j},
\]
we recover the standard synchronous RGS convergence bound
\[
\expectabbr{j+1} \leq \bigl(1-\beta(2-\beta)\mu\bigr)\expectabbr{j}
\]
as shown in~\cite{LL2010}.
Thus, the additional factor appearing in the denominator of the optimized bound should be viewed as a technical price to obtain a uniform estimate that also covers the delay-dominated case when \(\tau>0\), rather than as a deterioration inherent in the synchronous algorithm.
\end{remark}

%% file: discussion.tex
\section{Discussion}
\label{sec:discussion}
In this section, we provide a discussion of the analysis along with accompanying numerical experiments.
Our attention is restricted to scenarios of distributed memory systems.
The case of shared memory has already been studied in~\cite{ADG2015}.

To generate linear systems \(Ax=b\), we employ mainly three \(100\)-by-\(100\) test matrices as \(A\) to demonstrate our theoretical findings.
``Matrix 1'' and ``Matrix 2'' both have condition number \(\kappa(A) = 100\); the eigenvalues of Matrix 1 are arithmetically distributed, while those of Matrix 2 are randomly generated with a uniformly distributed logarithm. 
As a third test case, we use the two-dimensional \(5\)-point discrete negative Laplacian matrix, generated in MATLAB via the command \texttt{delsq(numgrid(`S', 12))}.
The right-hand side \(b\) is generated by the command \texttt{rand(100, 1)-0.5}.

To demonstrate the convergence rate stated in Theorem~\ref{thm:improved} and to gain further insight into the convergence behavior of the asynchronous Jacobi method, we employ MATLAB code based on the model~\eqref{eq:model} and~\cite{CM1969} to simulate the behavior of the asynchronous Jacobi method.
In the following, we consider the effects of \(\rho\), \(\tau\), and \(\beta\), respectively.

\subsection{Influence of \(\rho\)}
Recall the definition of \(\rho\), i.e.,
\[
\rho = \frac{1}{n}\max_{t=1,\dotsc,n}\bigl\{\sum_{r=1}^n\abs{\bigl(\bar{A}_{\rm{off}}\bigr)_{r,t}}\bigr\},
\]
where, in distributed memory systems, \(\bar{A}_{\rm{off}}\) consists of the off-diagonal blocks of \(\bar{A} = \sqrt{\Lambda}^{-1}A\sqrt{\Lambda}^{-1}\). 
Notice that \(\bar{A}_{r,t} \leq 1\) for all \(1\leq r\), \(t\leq n\).  
Consequently, one typically has \(\rho = \bigO(1/n)\), in particular when \(A\) is sparse.

In the distributed memory setting, the parameter \(\rho\) depends not only on the properties of the matrix \(A\), but also on the number of nodes and the manner in which the rows or columns of \(A\) are partitioned across these nodes.
Suppose that, as described in Algorithm~\ref{algo:async-Jacobi}, the \(i\)-th node stores \(x_{\Omega_i}\), \(b_{\Omega_i}\), and \(A_{\Omega_i, :}\), where \(p\) denotes the total number of nodes and \(\{\Omega_i\}_{i=1}^p\) constitutes a partition of the index set \(\{1, 2, \dotsc, n\}\), i.e., \(\cup_{i=1}^p \Omega_i = \{1, 2, \dotsc, n\}\).
Notably, there is an implicit form of local synchronization within each node.
In particular, when updating the \(k\)-th component of \(x\) at \((j+1)\)-st iteration, that is, when \(d^{(j)}=I_k\) in the model~\eqref{eq:model}, the most recent updates of other components within the same partition $\Omega_i$ (\(k\in \Omega_i\)) are always available and can be utilized.
Therefore, for a fixed matrix \(A\) of a given dimension, employing fewer nodes generally implies that each node stores a larger portion of the data.
This increases the extent of local synchronization, which in turn tends to reduce the value of $\rho$ and consequently accelerates convergence.

We now provide a clearer explanation of \(\rho\) and \(\bar{A}_{\mathrm{off}}\).
The matrix \(\bar{A}_{\mathrm{off}}\) depends not only on the structure of \(A\), but also on the synchronization pattern across nodes.
In particular, for \(A_{r,t} \neq 0\), whether \((\bar{A}_{\mathrm{off}})_{r,t}\) is zero depends on whether the variables \(x_r\) and \(x_t\) are synchronized across the nodes where they reside, either implicitly or explicitly.
For example, if \(x_r\) and \(x_t\) are stored on the same node, this corresponds to implicit synchronization, and thus \((\bar{A}_{\mathrm{off}})_{r,t} = 0\).
More generally, even if \(x_r\) and \(x_t\) are located on different nodes, \((\bar{A}_{\mathrm{off}})_{r,t}\) is still zero if local synchronization is performed between these nodes.
In the extreme case where global synchronization is enforced after each iteration, we have \(\bar{A}_{\mathrm{off}} = 0\), which implies \(\rho = 0\).
Therefore, by exploiting the sparsity structure of \(A\) to enable local synchronization, more entries of \(\bar{A}_{\mathrm{off}}\) can be driven to zero, thereby further reducing \(\rho\).
Moreover, in the definition~\eqref{eq:def-rho-omega} of \(\rho\), the value of \(\rho\) depends not only on the number of nonzero entries, but also on the magnitudes of those nonzero entries. 
This suggests that, in the distributed memory setting, since \(\rho\) affects convergence, it may be beneficial to use parallel data distributions that aim to reduce \(\rho\), rather than focusing solely on the traditional goal of strictly minimizing communication.

In Figure~\ref{fig:tau-rho}, we consider a \(100 \times 100\) Laplacian matrix under different numbers of nodes (denoted by \(p\)), with the rows of \(A\) distributed evenly across nodes.
For \(p = 5\), \(10\), and \(20\), each node stores \(20\), \(10\), and \(5\) rows of \(A\), respectively, and the corresponding values of \(\rho\) are \(2.5 \times 10^{-3}\), \(5.0 \times 10^{-3}\), and \(7.5 \times 10^{-3}\).
A comparison of the dashed (or solid) curves shows that a smaller value of \(\rho\) corresponds to a faster convergence rate.

\begin{figure}[!tb]\centering
\includegraphics[width=\textwidth]{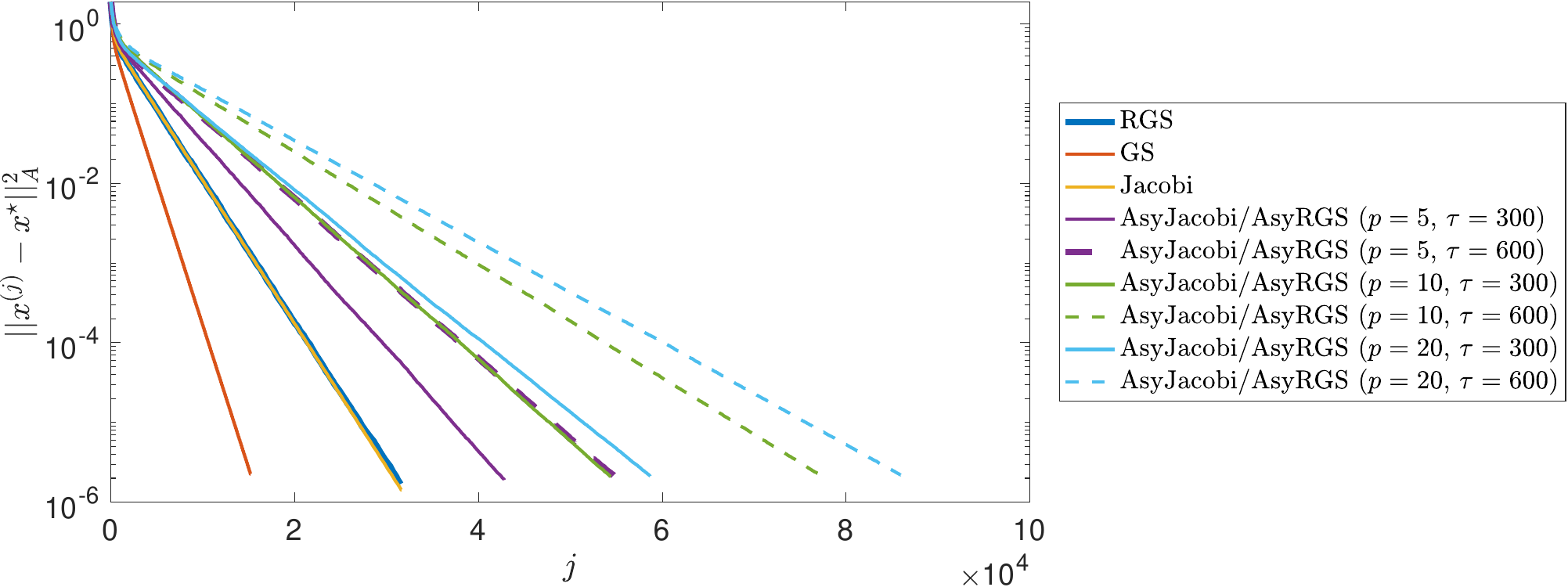}
\caption{Comparison of the Jacobi, RGS, GS, and AsyJacobi/AsyRGS methods for various values of the number of nodes and the delayed bound, i.e., \(p\) and \(\tau\) on the \(100\)-by-\(100\) Laplacian matrix. Here, AsyJacobi/AsyRGS refers to the asynchronous Jacobi/RGS method. Consistent with the notation used above, \(j\) denotes a single component update.}
\label{fig:tau-rho}
\end{figure}

\subsection{Influence of \(\tau\)}
The parameter \(\tau\) depends on the communication efficiency of the parallel system and represents delays that affect the convergence rate.
As \(\tau\) increases, the convergence rate deteriorates, as illustrated in Figure~\ref{fig:tau-rho}; in particular, this effect can be observed by comparing dashed and solid curves of the same color.

It should be noted that \(\tau\) is defined as an upper bound on the number of delayed iterations, rather than in terms of computational time.
In the distributed memory setting, \(\tau\) is primarily influenced by two factors.
The first is the number of iterations that can be executed simultaneously within a given time interval, denoted by \(\mathrm{iter}_t\).
The second is the number of time intervals required for the slowest node to transmit data to the node responsible for performing the current iteration, denoted by \(t_{\mathrm{comm}}\).
Accordingly, \(\tau \approx t_{\mathrm{comm}}\cdot \mathrm{iter}_t\).

The first factor mainly depends on the number of nodes, and one typically has \(\mathrm{iter}_t \approx p\), since a larger number of nodes allows more iterations to be carried out in parallel within the same time period.

The second factor depends on the number of nodes, the communication efficiency between nodes, and the underlying communication network topology.
For example, in \(2\)-D or \(3\)-D grid topologies, increasing the number of nodes generally increases the average communication distance, resulting in a larger \(t_{\mathrm{comm}}\) and hence a larger \(\tau\).
In contrast, for a fully connected network, increasing the number of nodes has a negligible impact on \(t_{\mathrm{comm}}\).
In addition, the sparsity structure of \(A\) and the way the data are distributed across nodes also influence \(t_{\mathrm{comm}}\).
In an ideal scenario, each row of \(A\) contains only \(\bigO(1)\) nonzero entries, i.e., the number of nonzeros per row is independent of both \(p\) and \(n\).
In this case, each node needs to communicate with at most \(\bigO(1)\) other nodes.
Consequently, with a suitable data distribution, the parameter \(t_{\mathrm{comm}}\) is unlikely to grow with either \(p\) or \(n\).

\subsection{Influence of \(\beta\)}

\begin{figure}[!tb]\centering
\includegraphics[width=\textwidth]{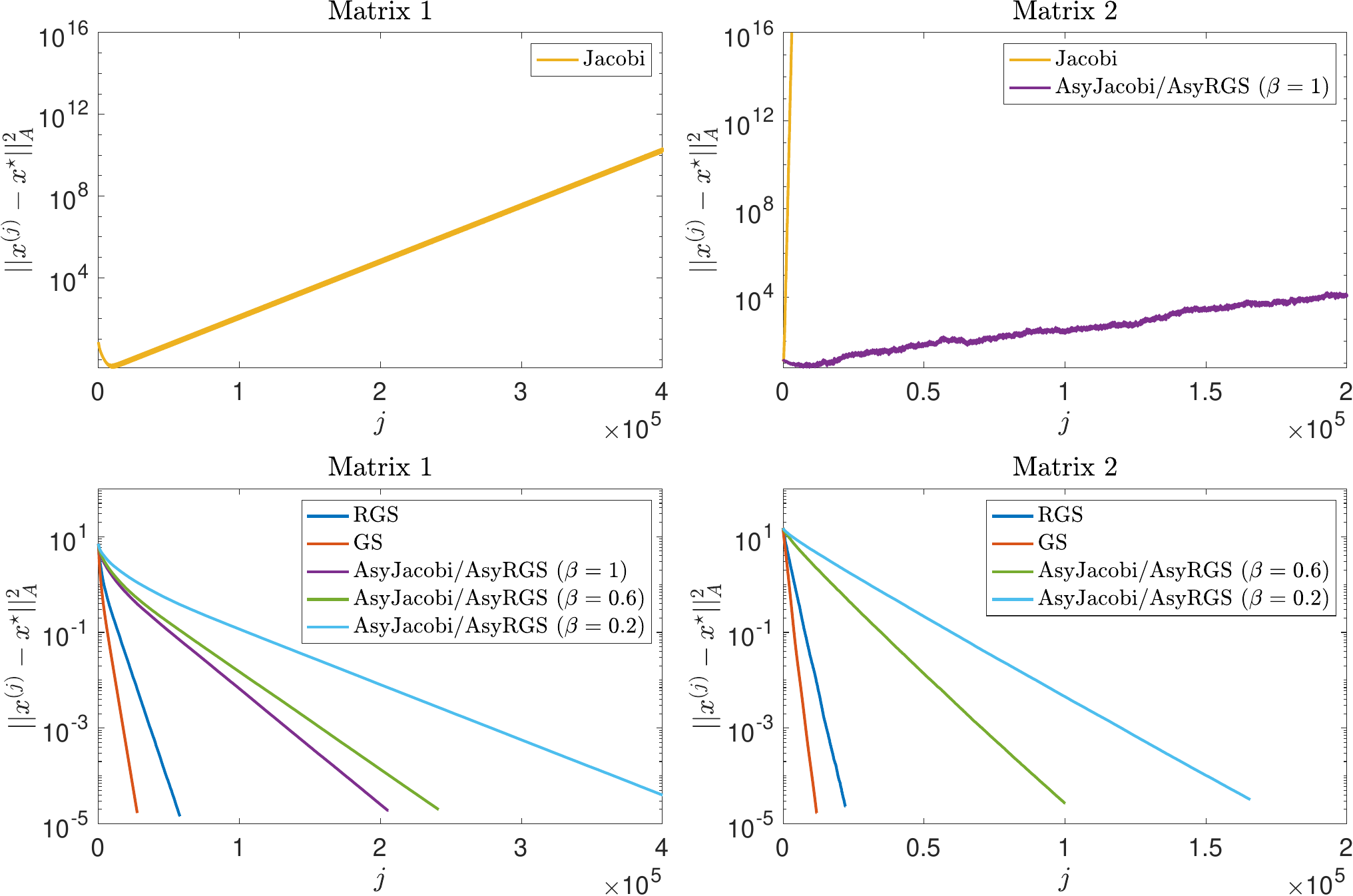}
\caption{Comparison of the Jacobi, RGS, GS, and AsyJacobi/AsyRGS methods for various values of \(\beta\) on ``Matrix 1'' and ``Matrix 2''. Here, AsyJacobi/AsyRGS refers to the asynchronous Jacobi/RGS method. Consistent with the notation used above, \(j\) denotes a single component update.}
\label{fig:div}
\end{figure}

Theorem~\ref{thm:improved} requires that \(\beta\) satisfies~\eqref{eq:beta-condition-new} to guarantee the convergence of the asynchronous method.
However, this requirement can be somewhat pessimistic, primarily due to the involvement of \(\tau\). 
The dependence on \(\tau\) arises from the use of the bound \(\abs{\{1,\dotsc, n\}\setminus\setphij{j}} \leq \tau\) in the analysis, which reflects a worst-case scenario rather than the average behavior observed in practice.
Nevertheless, despite this conservative estimation, the resulting condition remains informative.
In particular, it highlights that the choice of \(\beta\) should be closely related to \(\rho\) and \(\tau\).
This insight suggests that when \(\rho\) or \(\tau\) are large, selecting a smaller value of \(\beta\) can help ensure convergence of the asynchronous method.
This phenomenon is illustrated in the right sub-plot of Figure~\ref{fig:div}, where the asynchronous method diverges for \(\beta = 1\), but converges for smaller values such as \(\beta = 0.6\) or \(\beta = 0.2\).

Although a smaller value of \(\beta\) can improve stability and ensure convergence, it also affects the convergence rate.
In particular, the choice of \(\beta\) influences not only whether the asynchronous method converges, but also the convergence rate characterized in~\eqref{eq:thm-improved:convrate-new} and~\eqref{eq:cor-improved:convrate-new}.
According to~\eqref{eq:thm-improved:convrate-new} and~\eqref{eq:cor-improved:convrate-new}, an excessively small \(\beta\) leads to a slower convergence rate.
This trade-off is illustrated in Figure~\ref{fig:div}.
Although \(\beta = 0.6\) guarantees convergence for both ``Matrix 1'' and ``Matrix 2'', its convergence rate is slower than that of \(\beta = 1\) for ``Matrix 1''.
A similar behavior is observed for \(\beta = 0.2\), where convergence is still ensured, but the convergence rate is further reduced compared to \(\beta = 0.6\).

%% file: conclusion.tex
\section{Conclusions}
\label{sec:conclusions}
In this work, we have developed a unified framework for the convergence analysis of asynchronous Jacobi/RGS methods in shared memory and distributed memory environments.
The framework accommodates the different patterns of stale information arising in the two architectures and allows their effects on convergence to be studied within the same analysis.
For symmetric positive definite linear systems, we established linear convergence in expectation under an explicit stability condition and quantified the influence of communication delays and communication patterns on the resulting convergence rate.

An important distinction between our result and the convergence-rate analysis of~\cite{ADG2015} is the dependence on the delay bound \(\tau\).
Their analysis for asynchronous RGS in shared memory involves factors that depend exponentially on \(\tau\).
Such factors can lead to increasingly pessimistic bounds as the delay grows and are particularly unfavorable for distributed memory implementations, where communication delays can be considerably larger.
In contrast, the convergence bound obtained in our analysis removes these exponentially \(\tau\)-dependent factors, with the effect of the delay characterized through \(\sqrt{\rho\tau}+\rho\tau\).
This provides a more favorable description of the large-delay regime and makes explicit the joint roles of the delay bound and the communication structure.

The resulting theory also shows that, under an appropriate scaling regime with \(\rho\tau=\bigO(1)\), the guaranteed per-iteration convergence rate has the same asymptotic order as that of synchronous RGS.
Thus, although asynchronicity affects the constants in the convergence bound, it does not need to change the asymptotic convergence-rate order when the delay and communication structure are appropriately balanced.
This observation helps to clarify when asynchronous Jacobi/RGS methods can retain favorable convergence behavior in both shared memory and distributed memory environments.

%% file: acknowledgments.tex
\section*{Acknowledgments}
The authors thank Edmond Chow for providing the MATLAB script that simulates the asynchronous Jacobi method, following the approach in~\cite{CM1969}.

Both authors are supported by the European Union (ERC, inEXASCALE, 101075632). Views and opinions expressed are those of the authors only and do not necessarily reflect those of the European Union or the European Research Council. Neither the European Union nor the granting authority can be held responsible for them. Both authors additionally acknowledge support from the Charles University Research Centre program No. UNCE/24/SCI/005.